\documentclass[final]{siamltex}
\usepackage{graphicx}
\usepackage{psfrag}
\usepackage{float}
\usepackage{amssymb,amsmath}
\usepackage{t1enc}
\usepackage[latin1]{inputenc}
\usepackage{amsfonts}
\usepackage{stmaryrd}
\usepackage[bottom]{footmisc}
\usepackage{multicol}        
\usepackage{makeidx}         
\usepackage{showkeys}         
\graphicspath{{Figures/}}

\newtheorem{remark}[theorem]{Remark}

\newcommand{\ds}{\displaystyle}
\newcommand{\PP}{\mathbf{P}}
\newcommand{\NN}{\mathbb{N}}
\newcommand{\R}{\mathbb{R}}

\renewcommand{\phi}{\varphi}

\newcommand{\vecb}[1]{\pmb{#1}}

\newcommand{\nv}{\vecb{n}}
\newcommand{\av}{\vecb{b}}
\newcommand{\rv}{\vecb{r}}
\newcommand{\xv}{\vecb{\phi}}

\newcommand{\dtq}[1]{\frac{d#1}{dt}\,}

\newcommand{\uu}{u}

\newcommand{\uud}{U}

\renewcommand{\div}{\nabla\cdot}
\newcommand{\divs}[1]{\nabla_{#1}\cdot}
\newcommand{\divts}[1]{\widetilde{\nabla}_{#1}\cdot}
\newcommand{\grads}[1]{\nabla_{#1}}
\newcommand{\gradts}[1]{\widetilde{\nabla}_{#1}}
\newcommand{\inter}[1]{[\![#1]\!] } 
\newcommand{\trace}[1]{_{\,\mid#1}}
\newcommand{\moit}[1]{\frac{#1}{2}}
\newcommand{\co}{{\mathcal O}}

\newcommand{\gv}{\vecb{g}}

\numberwithin{equation}{section}

\begin{document}

\title{Optimized
  Schwarz waveform relaxation and discontinuous Galerkin time stepping for heterogeneous problems.}

 \author{Laurence Halpern \thanks{LAGA, Universit\'e Paris XIII,
 99 Avenue J-B Cl\'ement, 93430 Villetaneuse, France,
 \texttt{halpern@math.univ-paris13.fr}}
   \and  Caroline Japhet \thanks{LAGA, Universit\'e Paris XIII,
  99 Avenue J-B Cl\'ement, 93430 Villetaneuse, France, and
 CSCAMM, University of Maryland, College Park, MD 20742 USA,
 \texttt{japhet@cscamm.umd.edu}}
    \and J\'er\'emie Szeftel\thanks{D\'epartement de math\'ematiques et applications,
 Ecole Normale Sup\'erieure, 45 rue d'Ulm, 75230 Paris Cedex 05 France.
 \texttt{Jeremie.Szeftel@ens.fr}.
The first two authors are partially supported by french ANR (COMMA) and
GNR MoMaS}}

\maketitle

\begin{abstract}
We design and analyze Schwarz waveform relaxation algorithms for
domain decomposition of advection-diffusion-reaction problems with strong heterogeneities. These algorithms rely on optimized Robin or Ventcell transmission conditions, and can be used with curved interfaces. We analyze the semi-discretization in time with discontinuous Galerkin as well. We also show two-dimensional numerical results using generalized mortar finite elements in space.
\end{abstract}



\pagestyle{myheadings}
\thispagestyle{plain}
\markboth{L. Halpern, C. Japhet, J. Szeftel}{Schwarz Waveform relaxation and discontinuous Galerkin}


\section{Introduction}
In many fields of applications such as reactive transport, far field simulations of underground nuclear waste disposal or ocean-atmosphere coupling, models have to be coupled in different spatial zones, with very different space and time scales and possible complex geometries. For such problems with long time computations, a splitting of the time interval into windows is essential, with the possibility to use robust and fast solvers in each time window.

The Optimized Schwarz Waveform Relaxation (OSWR) method was introduced for linear parabolic and hyperbolic problems with constant coefficients in \cite{Gander:1999:OCO}. It was analyzed for advection diffusion equations, and applied to non constant advection, in \cite{martin:2005:AOS}. The algorithm computes independently in each subdomain over the whole time interval, exchanging space-time boundary data through optimized transmission operators. The operators are of Robin or Ventcell type, with coefficients optimizing a convergence factor, extending the strategy developed by F. Nataf and coauthors \cite{Charton:1991:MDD,japhet:1998:MDD}. The optimization problem was analyzed in \cite{gander:2007:OSW}, \cite{bennequin:2009:HBA}.

This method potentially applies to different space-time discretization in subdomains, possibly nonconforming and needs a very small number of iterations to converge. Numerical evidences of the performance of the method with variable smooth coefficients were given in \cite{martin:2005:AOS}. An extension to discontinuous coefficients was introduced in \cite{gander:2007:SWR}, with asymptotically optimized Robin transmission conditions in some particular cases.

The discontinuous Galerkin finite element method in time offers many advantages. Rigorous analysis can be made for any degree of accuracy and local time-stepping, and time steps can be adaptively controlled by a posteriori error analysis, see \cite{thomee:1997:FEM,johnson:1985:dgfirst,makridakis:2007:pea}. In a series of presentations in the regular domain decomposition meeting we presented the DG-OSWR method, using discontinuous Galerkin for the time discretization of the OSWR.
In \cite{blayo:2005:OSW}, \cite{halpern:2008:DGN}, we introduced the algorithm in one dimension with discontinuous coefficients.  In \cite{halpern:2009:DGN}, we extended the method to the two dimensional case. For the space discretization, we extended numerically the nonconforming approach in \cite{gander:2005:NZF} to advection-diffusion problems and optimized order 2 transmission conditions, to allow for non-matching grids in time and space on the boundary. The space-time projections between subdomains were computed with an optimal projection algorithm without any additional grid, as in \cite{gander:2005:NZF}. Two dimensional simulations were presented. In \cite{halpern:2010:DGN} we extended the proof of convergence of the OSWR algorithm to nonoverlapping subdomains with curved interfaces. Only sketches of proofs were presented.

The present paper intends to give a full and self-contained account of the method for the advection diffusion reaction equation with non constant coefficients.

In Section 2, we present the Robin and Ventcell algorithms at the continuous level in any dimension, and give in details the new proofs of convergence of the algorithms for nonoverlapping subdomains with curved interfaces.\\
Then in Section 3, we discretize in time with discontinuous Galerkin, and prove the convergence of the semi-discrete algorithms for flat interfaces. Error estimates are derived from the classical ones \cite{thomee:1997:FEM}. \\
The fully discrete problem is introduced in Section 4, using finite elements. The interfaces are treated by a new cement approach, extending the method in \cite{gander:2005:NZF}. Given the length of the paper, the numerical analysis will be treated in a forthcoming paper. \\
We finally present in Section 5 simulations for two subdomains, with piecewise smooth coefficients and a curved interface, for which no error estimates are available. We also include an application to the porous media equation.\\

%
Consider the advection-diffusion-reaction equation in $\Omega=\R^N$
\begin{equation}
  \label{eq:adall}
  \partial_t u +\div(\av u -\nu \nabla u)+c u=f  \quad \mbox{in } \Omega \times (0,T),
\end{equation}
with initial condition
\begin{equation}
  \label{eq:adu0}
  u(0,x)=u_0(x)  \quad x \in \Omega.
\end{equation}
The advection and diffusion coefficients $\av$ and $\nu$ , as well as
the reaction coefficient $c$, are piecewise smooth, the problem is
parabolic, i.e. $\nu \ge \nu_0 > 0$ \textit{a.e.} in $\R^N$ .
\begin{theorem}[Well-posedness and regularity, \cite{lions:1968:PAL}]\label{th,regall}
  Let $\Omega=\R^N$.  Suppose $\av\in (W^{1,\infty}(\Omega))^N$,
  $\nu \in W^{1,\infty}(\Omega)$ and $c \in L^{\infty}(\Omega)$. If
  the initial value $u_0$ is in $H^1(\Omega)$, and the right-hand side
  $f$ is in $L^2(0,T;L^2(\Omega))$, then there exists a unique solution $u$ of \eqref{eq:adall}, \eqref{eq:adu0} in
  $H^1(0,T;L^2(\Omega))$ $\cap$ $L^{\infty}(0,T;H^1(\Omega))$ $\cap$
  $L^{2}(0,T;H^2(\Omega))$.
\end{theorem}

We consider now a decomposition of $\Omega$ into nonoverlapping
subdomains $\Omega_i, i\in \inter{1,I}$, as depicted in Figure
\ref{fig:decompb}.  In all cases the boundaries
between the subdomains are supposed to be hyperplanes at infinity.
\begin{figure}[H]
  \centering
  \psfrag{Oj}{$ \Omega_i$}
  \begin{tabular}{cc}
  \includegraphics[width=0.17\textwidth]{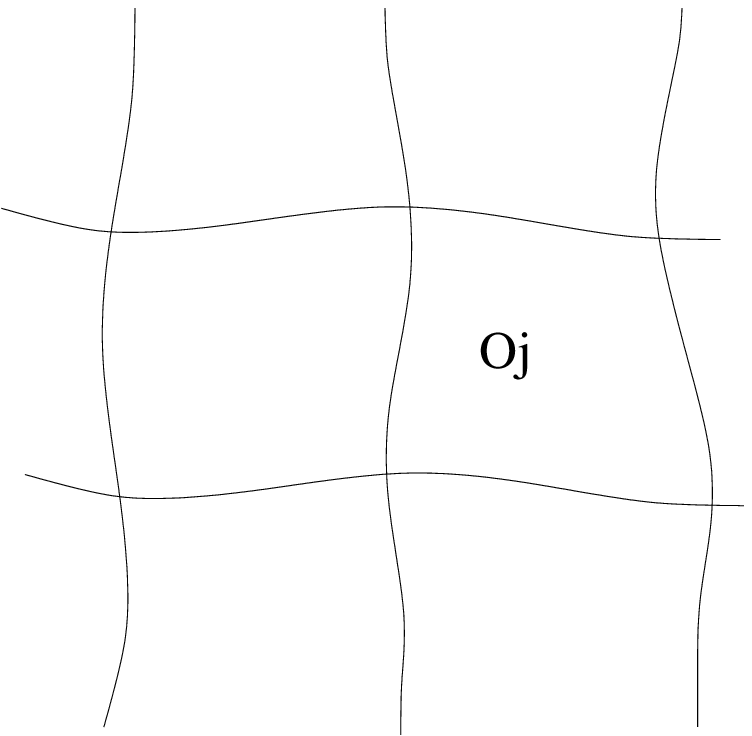} \hspace{2cm}
  &
   \includegraphics[width=0.14\textwidth]{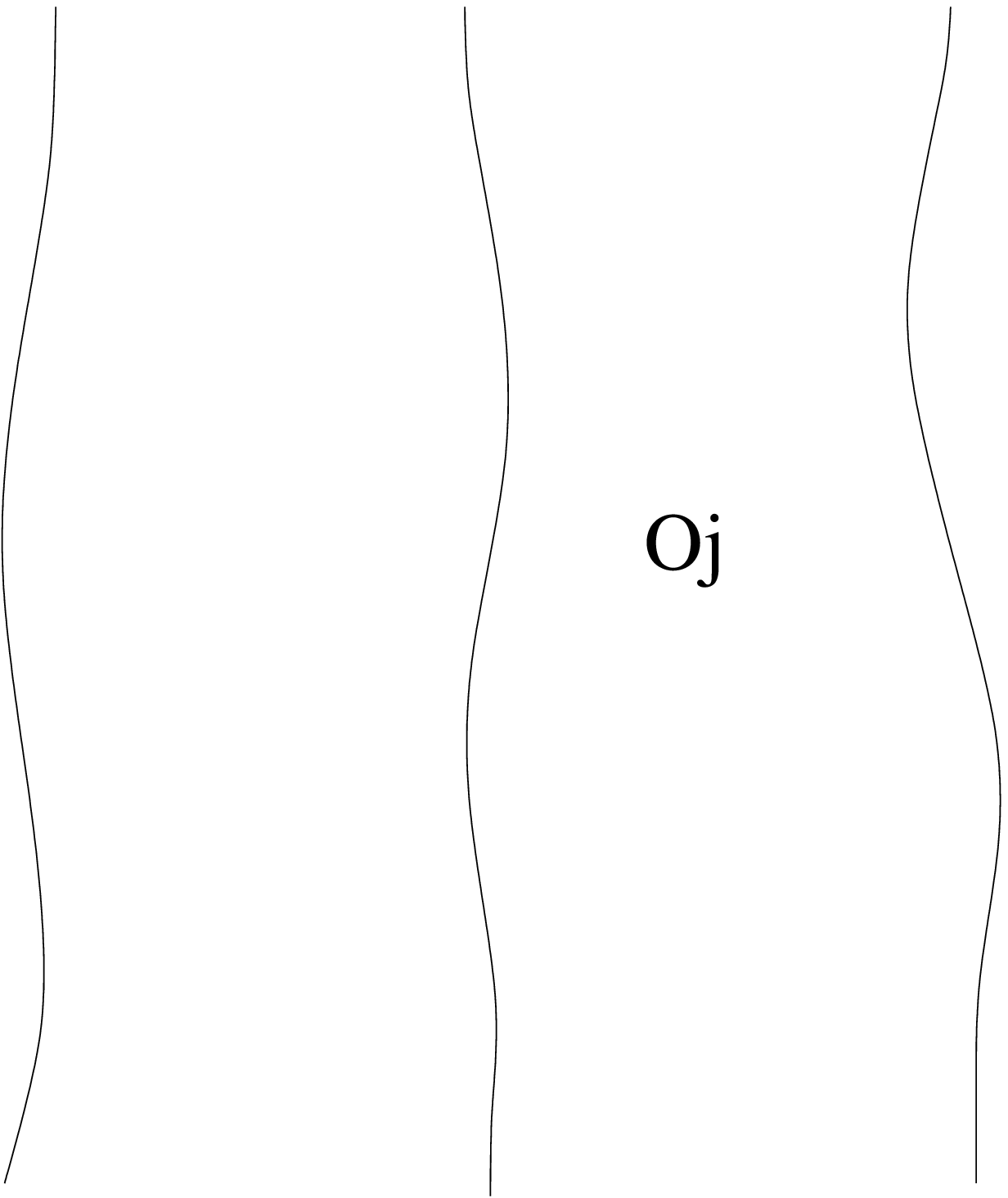}
   \end{tabular}
    \caption{Left: decomposition with possible corners (Robin
      transmission conditions), right: decomposition in bands (Ventcell transmission conditions)}
  \label{fig:decompb}
\end{figure}
Problem \eqref{eq:adall} is equivalent to solving $I$ problems in
subdomains $\Omega_i$, with transmission conditions on the interface
$\Gamma_{i,j} $ between two neighboring subdomains $\Omega_i$ and $\Omega_j$, given by the jumps $[u]=0$, $[(\nu\nabla u-\av)\cdot\nv_i]=0$. Here $\nv_i$ is the unit exterior normal to $\Omega_i$. As coefficients $\nu$ and $\av$ are possibly discontinuous on
the interface, we note, for $s \in \Gamma_{i,j}$,
$\nu_i(s)=\lim_{\varepsilon\rightarrow 0_+}\nu(s-\varepsilon \nv_i)$.
The same notation holds for $\av$ and $c$.

To any $i \in \inter{1,m}$, we associate the set ${\cal
  N}_i$ of indices of the neighbors of $\Omega_i$.

Following \cite{Charton:1991:MDD,Gander:1999:OCO,gander:2007:SWR}, we propose as preconditioner for (\ref{eq:adall}, \ref{eq:adu0}), the sequence of problems
\begin{subequations}\label{eq:algoorder2}
  \begin{align}
     &\partial_t u_i^k +\div(\av_i u_i^k-\nu_i\nabla u_i^k)+c_i u_i^k =f
      \mbox{ in }\Omega_i\times(0,T),\label{eq:algorogral1}\\[3mm]
     &\bigl(\nu_i \partial_{\nv_i}-\av_i\cdot\nv_i
      \bigr)\,u_i^{k} +\mathcal{S}_{i,j}u_i^{k} =
      \bigl(\nu_{j} \partial_{\nv_{i}}-\av_{j}\cdot
      \nv_{i}\bigr)\, u_{j}^{k-1} +\mathcal{S}_{i,j}u_{j}^{k-1}
      \mbox{ on }\Gamma_{i,j}, \, j\in {\cal N}_i. \label{eq:algogal2}
  \end{align}
\end{subequations}
The boundary operators $\mathcal{S}_{i,j}$, acting on the part $\Gamma_{i,j}$ of the boundary of $\Omega_i$ shared by the boundary of $\Omega_j$ are given by
\begin{equation}
  \label{eq:boundaryoperators}
     \mathcal{S}_{i,j}\phi=p_{i,j}\phi
    + q_{i,j} (\partial_t\phi+\divs{\Gamma_{i,j}}(\rv_{i,j}
    \phi  -s_{i,j} \grads{\Gamma{i,j}}\phi)).
\end{equation}
$\grads{\Gamma}$ and $\divs{\Gamma}$ are respectively the gradient
and divergence operators on $\Gamma$. $p_{i,j}, \ q_{i,j}, \ s_{i,j}$
are functions in $L^\infty(\Gamma_{i,j})$ and $\rv_{i,j}$ is in
$(L^\infty(\Gamma_{i,j}))^{N-1}$. The initial value is that of $u_0$ in each subdomain.  An initial guess $(g_{i,j})$ is given on
$L^2((0,T)\times \Gamma_{i,j})$ for $i\in \inter{1,I}, j \in {\cal
  N}_i$.  By convention for the first iterate, the right-hand side in \eqref{eq:algoorder2} is given by $g_{i,j}$.
Under regularity assumptions,
solving \eqref{eq:adall} is equivalent to solving
\begin{equation}\label{eq:dec}
\begin{array}{l}
    \partial_t u_i +\div(\av_i u_i-\nu_i\nabla u_i)+c_i u_i=f \mbox{ in } \Omega_i
    \times(0,T),  \\[0.5mm]
    \ds\bigl(\nu_i \partial_{\nv_i}
    -\av_i\cdot\nv_i \bigr)\, u_i +\mathcal{S}_{i,j}u_i =
    \bigl(\nu_{j}
    \ds\partial_{\nv_i} -\av_{j}\cdot\nv_{i} \bigr)\, u_{j}
    +\mathcal{S}_{i,j}u_{j} \mbox{ on }\Gamma_{i,j}\times (0,T), j\in {\cal N}_i,
\end{array}
\end{equation}
for $i\in \inter{1,I}$ with $u_i$ the restriction of $u$ to $\Omega_i$.

%
\section{Studying the algorithm for the P.D.E}
The first step of the study is to give a frame for the definition of the iterates.
\subsection{The local problem}
The optimized Schwarz waveform relaxation algorithm relies on the
resolution of the following initial boundary value problem in a
domain $\co$ with boundary $\Gamma$:
\begin{equation}\label{eq:pblim}
  \begin{array}{l}
     \partial_t w +\div(\av w -\nu \nabla w)+c w=f \mbox{ in }
    \co \times (0,T),\\
    \ds\nu\,\partial_{\nv} w - \av\cdot\nv \, w
    + \mathcal{S} w =g \mbox{ on }
    \Gamma \times (0,T),\\
    w(\cdot,0)=u_0 \mbox{ in }    \co,
  \end{array}
\end{equation}
where $\nv$ is the exterior unit normal to $\co$. The boundary operator $\mathcal{S}$ is defined on $\Gamma=\partial \co$ by
\begin{equation}\label{eq:boundaryoperator}
  \mathcal{S}w= p\,w + q\, (\partial_t w
  + \divs{\Gamma}(\rv w -s\grads{\Gamma}w)).
\end{equation}
The domain $\co$ has either form depicted in Figure \ref{fig:decompb}, left for $q=0$, right otherwise.

The functions $p,\,q$ and $s$ are in $L^\infty(\Gamma)$, and $\rv$ is in $(L^\infty(\Gamma))^{N-1}$. Either $q=0$, and the boundary condition is
of Robin type, or we suppose $q\ge q_0 >0$ and the operator will be referred to as Order 2 or Ventcell operator.  In the latter case, we
need the spaces
\begin{equation}\label{eq:sobtrace}
  H^s_s(\co)=\{v\in H^s(\co), v\trace{\Gamma}\in H^s(\Gamma)\},
\end{equation}
which are defined for $s>1/2$, and equipped with the scalar product
\begin{equation}\label{eq:sobtraceps}
  (w,v)_{H^s_s(\co)}=(w,v)_{H^s(\co)}+(qw,v)_{H^s(\Gamma)}.
\end{equation}

We define the bilinear forms $m$ on $H^1(\co)$ and $a$ on $ H^1_1(\co)$ by
\begin{equation}\label{eq:sobtraceps0}
 m(w,v)=(w,v)_{L^2(\co)}+(qw,v)_{L^2(\Gamma)},
\end{equation}
and
\begin{multline}
  \label{eq:formbilin}
  a(w,v):=
  \int_\co(\frac{1}{2}((\av\cdot\nabla w)v-(\av\cdot\nabla v)w))\, dx
  +\int_\co \nu \nabla w\cdot\nabla v \, dx + \int_\co (c+\frac{1}{2}\div\av) w v \, dx\\
  +\int_\Gamma \bigl((p-\frac{\av\cdot\nv}{2}+{q \over 2}\nabla_{\Gamma}
  \cdot\rv)w v+{q \over 2}(\nabla_{\Gamma} \cdot(\rv w)v-\nabla_{\Gamma} \cdot(\rv v)w)
  +qs\nabla_{\Gamma}w \cdot\nabla_{\Gamma} v \bigr)\, d \sigma,
\end{multline}
By the Green's formula, we can write a variational formulation of
\eqref{eq:pblim}:
\begin{equation}\label{eq:formvariat}
 \ds \frac{d\,}{dt}\,m(w, v)+a(w,v)
   = (f,v)_{L^2(\co)}+(g,v)_{L^2(\Gamma)}.
\end{equation}
The well-posedness is a generalization of results in
\cite{gander:2007:OSW,bennequin:2009:HBA,szeftel:2005:ABC}. It relies
on energy estimates and Gr\"{o}nwall's lemma.
\begin{theorem}\label{th:exfortefo}
  Suppose $\nu \in W^{1,\infty}(\co)$, $\av \in
  (W^{1,\infty}(\co))^N$, $c \in L^{\infty}(\co)$, $p \in L^{\infty}(\Gamma), \ q \in
  L^{\infty}(\Gamma)$, $\rv \in (W^{1,\infty}(\Gamma))^{N-1}$, $s \in
  W^{1,\infty}(\Gamma)$ with $s >0$ \textit{a.e}.

   \textbf{If $q=0$}, if $f$ is in $H^{1}(0,T;L^{2}(\co))$, $u_0$ is
   in $H^{2}(\co)$   and $g$ is in $H^1(0,T;L^2(\Gamma))\cap L^{\infty}(0,T;H^{1/2}(\Gamma))$, satisfying the compatibility condition
   $\ds\nu\,\partial_{\nv} u_0- \av\cdot\nv \, u_0
   + p u_0 =g$
, the
   subdomain problem \eqref{eq:pblim} has a unique solution $w$ in
   $L^\infty(0,T;H^{2}(\co))$ $\cap$ $W^{1,\infty}(0,T;L^2(\co))$.

  \textbf{If $q\ge q_0 >0 \ a.e.$}, if $f$ is in
  $H^{1}(0,T;L^{2}(\co))$, $u_0$ is in $H^{2}_2(\co)$, and $g$ is in
  $H^{1}((0,T);L^2(\Gamma))$, problem \eqref{eq:pblim}
  has a unique solution $w$ in $L^\infty(0,T;H^{2}_2(\co))$ $\cap$ $
  W^{1,\infty}(0,T;L^2(\co))$ with $\partial_tw\in L^\infty
  (0,T;L^2(\Gamma))$.
\end{theorem}
\begin{proof}
The existence result relies on a Galerkin method like in
\cite{szeftel:2003:ABC,szeftel:2005:ABC}. In the sequel, $\alpha,\beta,\cdots$ denote positive real  numbers depending only on the coefficients and the geometry. The basic
estimate is obtained by multiplying the equation by $w$ and
integrating by parts in the domain.  We set $\|w\|=\|w\|_{L^2(\co)}$
and $\|w\|_\Gamma=\|w\|_{L^2(\Gamma)}$.
\begin{equation*}
\ds \frac{1}{2}\frac{d}{dt} m(w,w)\ +\
  a(w,w)= (f,w)_{L^2(\co)} + (g,w)_{L^2(\Gamma)}.
\end{equation*}
With the assumptions on the coefficients, we have \\

\textbf{Case $q=0$.}
\begin{equation*}
\begin{array}{lcl}
  a(w,w)&=&\ds\int_\co \nu |\nabla w|^2 \, dx
           + \int_\co (c+\frac{1}{2}\div\av)w^2 \, dx
           + \int_\Gamma
                \bigl(
                 (p -\frac{\av\cdot\nv}{2}
                 \bigr) w^2\, d \sigma \\[3mm]
        &\ge &\ds \alpha(\|\nabla w\|^2
                -\beta (\|w\|^2+\|w\|_\Gamma^2)\\[3mm]
       &\ge &\ds \frac{\alpha}{2}\|\nabla w\|^2
                -\gamma \|w\|^2,
  \end{array}
\end{equation*}
the last inequality coming from the trace theorem
\begin{equation}\label{eq:tracetheorem}
\|w\|_\Gamma^2 \le C \|\nabla w\| \|w\|.
\end{equation}
 We obtain with the Cauchy-Schwarz inequality
\begin{equation*}
\begin{array}{lcl}
 \frac{1}{2} \frac{d}{dt} \|w\|^2\ +\
  \frac{\alpha}{4}\|\nabla w\|^2
               &\le& \eta \|w\|^2 +
          \delta (\|f\|^2+\|g\|_\Gamma^2).
  \end{array}
\end{equation*}
We now have with Gr\"{o}nwall's lemma
\begin{equation}\label{eq:estimatepblimRobin}
\begin{array}{l}
    \ds \|w(t)\|^2\ +\
    \frac{\alpha}{2}\int_0^t(\|\nabla w(s)\|^2 ds  \le\\[3mm]
    \hspace{10mm}
   \ds  e^{2\eta  T}
        ( \|u_0\|^2 +
              2 \delta (\|f\|_{L^2(0,T;L^2(\co))}^2
                        +\|g\|_{L^2(0,T;L^2(\Gamma))}^2)).
  \end{array}
\end{equation}
We apply \eqref{eq:estimatepblimRobin} to $w_t$:
\begin{equation*}
\begin{array}{l}
    \ds \|w_t(t)\|^2\ +
     \frac{\alpha}{2}\int_0^t\|\nabla w_t(s)\|^2 ds  \le\\[3mm]
    \hspace{10mm}
   \ds  e^{2\eta  T}
        ( \|{w_t}_0\|^2 +
              2 \delta (\|f_t\|_{L^2(0,T;L^2(\co))}^2
                        +\|g_t\|_{L^2(0,T;L^2(\Gamma))}^2)).
  \end{array}
\end{equation*}
Thanks to the compatibility condition, ${w_t}_0$ can be estimated,
using the equation, by $\|{w_t}_0\| \le \zeta(\|u_0\|_{H^2(\co)} +
\|f(0,\cdot)\|)$ , and we obtain
\begin{equation*}
\begin{array}{l}
    \ds \|w_t(t)\|^2\ +
    \frac{\alpha}{2}\int_0^t\|\nabla w_t(s)\|^2 ds  \le\\[3mm]
    \hspace{10mm}
   \ds \theta e^{2\eta  T}
        ( \|u_0\|_{H^2(\co)}  +
                (\|f\|_{H^1(0,T;L^2(\co))}^2
                        +\|g\|_{H^1(0,T;L^2(\Gamma))}^2)).
  \end{array}
\end{equation*}
\textbf{Case $q\ge q_0 >0$} \ \textit{a.e}.
\begin{equation*}
\begin{array}{lcl}
  a(w,w)&=&\ds\int_\co \nu |\nabla w|^2 \, dx
           + \int_\co (c+\frac{1}{2}\div\av)w^2 \, dx \\[2mm]
           & & \ds\quad
           + \int_\Gamma
                \bigl(
                 (p -\frac{\av\cdot\nv}{2}+
                 {q \over 2}\nabla_{\Gamma}\cdot\rv)w^2
                 +qs |\nabla_{\Gamma}w|^2
                 \bigr) \, d \sigma,\\[2mm]
        &\ge &\ds \alpha(\|\nabla w\|^2
               + \|\nabla_{\Gamma} w\|_{\Gamma}^2)
               -\beta m(w,w),
  \end{array}
\end{equation*}
and by the Cauchy-Schwarz inequality
\begin{equation*}
\int_\co (f,w)\,dx +
           \int_\Gamma (g,w)_\Gamma\,d\sigma
           \le
           \gamma m(w,w)
           + \delta (\|f\|^2+\|g\|_\Gamma^2).
\end{equation*}
Collecting these inequalities, we obtain
\begin{equation*}
\begin{array}{lcl}
 \frac{1}{2}\frac{d}{dt} m(w,w)\ +\
  \alpha(\|\nabla w\|^2
               + \|\nabla_{\Gamma} w\|_{\Gamma}^2)
               &\le& (\beta+\gamma) m(w,w) +
          \delta (\|f\|^2+\|g\|_\Gamma^2).
  \end{array}
\end{equation*}
We now integrate in time and use Gr\"{o}nwall's lemma to obtain for any $t$ in $(0,T)$
\begin{equation}\label{eq:estimatepblim}
\begin{array}{l}
    \ds m(w(t),w(t))\ +\
    2\alpha\int_0^t(\|\nabla w(s)\|^2
               + \|\nabla_{\Gamma} w(s)\|_{\Gamma}^2)ds  \le\\[3mm]
    \hspace{10mm}
   \ds  e^{(\beta+\gamma) T}
        ( m(u_0,u_0) +
              2 \delta (\|f\|_{L^2(0,T;L^2(\co))}^2
                        +\|g\|_{L^2(0,T;L^2(\Gamma))}^2)).
  \end{array}
\end{equation}
We apply \eqref{eq:estimatepblim} to $w_t$:
 \begin{equation*}
\begin{array}{l}
    \ds m(w_t(t),w_t(t))\ +\
    2\alpha\int_0^t(\|\nabla w_t(s)\|^2
               + \|\nabla_{\Gamma} w_t(s)\|_{\Gamma}^2)ds  \le\\[3mm]
    \hspace{10mm}
   \ds  e^{(\beta+\gamma) T}
        ( m({w_t}_0,{w_t}_0) +
              2 \delta (\|f_t\|_{L^2(0,T;L^2(\co))}^2
                        +\|g_t\|_{L^2(0,T;L^2(\Gamma))}^2) ).
  \end{array}
  \end{equation*}
We now use the equations at time 0 to estimate $m({w_t}_0,{w_t}_0)$.
From the equation in the domain, we deduce that
\[
\|{w_t}_0\|\le \zeta(\|{u}_0\|_{H^2(\co)}+ \|f(0,\cdot)\|),
\]
and from the boundary condition that
\[
\|{w_t}_0\|_{\Gamma}\le \eta(\|{u}_0\|_{H_2^2(\co)}+ \|g(0,\cdot)\|_{\Gamma}),
\]
which gives altogether
 \begin{equation}\label{eq:estimatepblim2}
\begin{array}{l}
    \ds m(w_t(t),w_t(t))\ +\
    2\alpha\int_0^t(\|\nabla w_t(s)\|^2
               + \|\nabla_{\Gamma} w_t(s)\|_{\Gamma}^2)ds  \le\\[3mm]
    \hspace{10mm}
   \ds  \theta e^{(\beta+\gamma) T}
        ( \|{u}_0\|_{H_2^2(\co)}^2+
          \|f\|_{H^1(0,T;L^2(\co))}^2
                        +\|g\|_{H^1(0,T;L^2(\Gamma))}^2) ).
  \end{array}
  \end{equation}
We can now apply the Galerkin method.  When $q=0$, we work in $H^1(0,T;H^{1}(\co))$ $\cap $ $W^{1,\infty}(0,T;L^2(\co))$ , while
if $q\geq q_0>0$ \textit{a.e} we consider $H^1(0,T;H^{1}_1(\co))$ $\cap$ $ W^{1,\infty}(0,T;L^2(\co))$ $\cap$ $W^{1,\infty}(0,T;L^2(\Gamma))$. This gives a unique solution $w$. The regularity
$H^2$ is obtained for $q=0$ by the usual regularity results for the
Laplace equation with Neumann boundary condition, since
\begin{equation*}
\begin{array}{l}
-\Delta w= \ds \frac{1}{\nu}
(f -w_t -\div(\av w)+ \nabla \nu \cdot \nabla w)
\in L^\infty(0,T;L^2(\co)),\\
\ds  \partial_{\nv} w =\ds \frac{1}{\nu}(\av\cdot\nv  -p ) w
+\frac{1}{\nu}g  \in L^\infty(0,T;H^{1/2}(\Gamma)).
\end{array}
\end{equation*}
In the other case, we have that
\begin{equation*}
\begin{array}{l}
-\Delta w= \ds \frac{1}{\nu}
(f -w_t -\div(\av w)+ \nabla \nu \cdot \nabla w) \in L^\infty(0,T;L^2(\co)),\\
\ds \nu \partial_{\nv}- qs\Delta_\Gamma w=\ds \frac{1}{\nu}(\av\cdot\nv -p
-q\, (\partial_t + \divs{\Gamma}\rv -(\divs{\Gamma}s)\grads{\Gamma}))) w
+\frac{1}{\nu}g
   \in L^\infty(0,T;L^{2}(\Gamma)),
\end{array}
\end{equation*}
and we conclude like in \cite{szeftel:2003:ABC,szeftel:2005:ABC}.
\end{proof}
\subsection{Convergence analysis for Robin transmission conditions}
We suppose here the coefficients $q_i$ to be zero everywhere. Given
initial guess $(g_{i,j})$ on $L^2((0,T)\times \Gamma_{i,j})$ for $i\in
\inter{1,I}, j \in {\cal N}_i$, the algorithm reduces in each
subdomain to
\begin{subequations}\label{eq:algorobin}
  \begin{align}
     &\partial_t u_i^k +\div(\av_i u_i^k-\nu_i \nabla u_i^k)+c_i u_i^k =f
      \mbox{ in }\Omega_i\times(0,T),\label{eq:algorobin1}\\[3mm]
     &\bigl(\nu_i \partial_{\nv_i}-\av_i\cdot\nv_i
      \bigr)\,u_i^{k} +p_{i,j}\,u_i^{k} =
      \bigl(\nu_{j} \partial_{\nv_{i}}-
      \av_j\cdot\nv_i\bigr)\, u_{j}^{k-1} +p_{i,j}\,u_{j}^{k-1}
      \mbox{ on }\Gamma_{i,j}, j\in {\cal N}_i. \label{eq:algorobin2}
  \end{align}
\end{subequations}
The well-posedness for the boundary value problem in the previous
section permits to define the sequence of iterates. We now consider
the convergence of this sequence.
\begin{theorem}\label{th:convcont}
  For coefficients $p_{i,j}$ such that
  $p_{i,j}+p_{j,i}>0 \ a.e.$, the sequence
  $(u_i^k)_{k\in \NN}$ of solutions of \eqref{eq:algorobin} converges
  to the solution $u$ of problem \eqref{eq:adall}.
\end{theorem}
\begin{proof}
By linearity, it is sufficient to prove that the sequence of iterates
converges to zero if $f=u_0=0$.

We multiply \eqref{eq:algorobin1} by $u_i^k$, integrate on $\Omega_i$,
and use the Green's formula. We obtain
\begin{multline}\label{eq:estcont1}
  \frac{1}{2}\frac{d\,}{dt}\|u_i^k(t,\cdot)\|^2_{L^2(\Omega_i)}
  +(\nu_i\nabla u_i^k,\nabla u_i^k)_{L^2(\Omega_i)}
+ ((c_i+\frac{1}{2}\div\av_i) u_i^k,u_i^k)_{L^2(\Omega_i)} \\ \hspace{-3cm}-\sum_{j\in{\cal
  N}_i}\int_{\Gamma_{i,j}} \bigl(\nu_i \partial_{\nv_i} u_i^k
  -\moit{\av_i\cdot\nv_i} u_i^k\bigr)\,u_i^k d\sigma = 0.
\end{multline}
We use now
\begin{multline}\label{tototiti}
  \bigl(\nu_i \partial_{\nv_i} u_i^k-\av_i\cdot\nv_i
  u_i^k+p_{i,j}u_i^k\bigr)^2
  -\bigl(\nu_i \partial_{\nv_i} u_i^k-\av_i\cdot\nv_i
  u_i^k-p_{j,i}u_i^k\bigr)^2 =\\
  2(p_{i,j}+p_{j,i})\bigl(\nu_i \partial_{\nv_i} u_i^k
  -\moit{\av_i\cdot\nv_i} u_i^k \bigr)u_i^k
  + (p_{i,j}+p_{j,i})(p_{i,j}-p_{j,i}-\av_i\cdot\nv_i)(u_i^k)^2.
\end{multline}
We replace the boundary term in \eqref{eq:estcont1}, and integrate in
time. Since the initial value vanishes, we have for any time $t$,
\begin{multline*}
  \|u_i^k(t)\|^2_{L^2(\Omega_i)}
  +2\int_{0}^t\bigl((\nu_i\nabla u_i^k,\nabla u_i^k)_{L^2(\Omega_i)}
       + ((c_i+\frac{1}{2}\div\av_i) u_i^k,u_i^k)_{L^2(\Omega_i)}\bigr) \,d\tau\\
  +\sum_{j\in{\cal N}_i}\int_{0}^t\int_{\Gamma_{i,j}}
\frac{1}{p_{i,j}+p_{j,i}}
  \bigl(\nu_i \partial_{\nv_i} u_i^k-{\av_i\cdot\nv_i}
  u_i^k-p_{j,i}u_i^k\bigr)^2 d\sigma \,d\tau \\
  =
  \sum_{j\in{\cal
  N}_i}\int_{0}^t\int_{\Gamma_{i,j}}
  \frac{1}{p_{i,j}+p_{j,i}}
  \bigl(\nu_i \partial_{\nv_i} u_i^k-{\av_i\cdot\nv_i}
  u_i^k+p_{i,j}u_i^k\bigr)^2d\sigma \,d\tau \\
  \hspace{-3cm}+\sum_{j\in{\cal
  N}_i}\int_{0}^t\int_{\Gamma_{i,j}}(p_{j,i}-p_{i,j} -\av_i\cdot\nv_i)
  (u_i^k)^2d\sigma \,d\tau.
\end{multline*}
Since the coefficients are all bounded, the last term in the
right-hand side can be handled by the trace theorem
\eqref{eq:tracetheorem} to be canceled with the terms in the
left-hand side like in the proof of Theorem \ref{th:convcont}.
%
We further insert the transmission condition in the right-hand side:
\begin{multline*}
  \|u_i^k(t)\|^2_{L^2(\Omega_i)}+\nu_0\int_{0}^t\|\nabla u_i^k\|^2_{L^2(
  \Omega_i)}\,d\tau
  +\sum_{j\in{\cal N}_i}\int_{0}^t\int_{\Gamma_{i,j}}
  \frac{1}{p_{i,j}+p_{j,i}}
  \bigl(\nu_i \partial_{\nv_i} u_i^k-\av_i\cdot\nv_i
  u_i^k-p_{j,i}u_i^k\bigr)^2 d\sigma\,d\tau \\
   \le
  \sum_{j\in{\cal N}_i}\int_{0}^t\int_{\Gamma_{i,j}}
  \frac{1}{p_{i,j}+p_{j,i}}
  (\nu_{j} \partial_{\nv_i} u_j^{k-1}-\av_{j}\cdot
  \nv_{i} u_{j}^{k-1}\, +p_{i,j}u_{j}^{k-1} \bigr)^2d\sigma\,d\tau +
  C_1\int_{0}^t\|u_i^k\|^2_{L^2(\Omega_{i})}.
\end{multline*}
We sum on the subdomains, and on the iterations, the boundary terms
cancel out except the first and last ones, and we obtain for any $t
\in (0,T)$,
\begin{multline}\label{eq:estcontal}
  \sum_{k\in \inter{1,K}}\sum_{i\in \inter{1,I}}\biggl(
  \|u_i^k(t)\|^2_{L^2(\Omega_i)}+\nu_0\int_{0}^t\|\nabla u_i^k\|^2_{L^2(
  \Omega_i)}\,d\tau
  \biggr)
  \le
  \alpha(t)
  + C_1\sum_{k\in \inter{1,K}}\sum_{i\in \inter{1,I}}\int_{0}^t\|u_i^k\|^2_{
  L^2(\Omega_{i})},
\end{multline}
with
\[
  \alpha(t)=
  \sum_{i\in \inter{1,I}} \sum_{j\in{\cal N}_i}
  \int_{0}^t\int_{\Gamma_{i,j}}
  \frac{1}{p_{i,j}+p_{j,i}}
  (\nu_{j} \partial_{\nv_i} u_{j}^{0}-\av_{j}\cdot
  \nv_{i} u_{j}^{0}\, +p_{i,j}u_{j}^{0} \bigr)^2 d\sigma\,d\tau.
\]
We now apply Gr\"{o}nwall's lemma and obtain that for any $K > 0$,
\[
 \ds\sum_{k\in \inter{1,K}}\sum_{i\in
  \inter{1,I}}\|u_i^k(t)\|^2_{L^2(\Omega_i)}
  \le
  \alpha(T) e^{C_1 T},
  \]
%
%
%
which proves that the sequence $u_i^k $ converges to zero in $L^2 ((0,T)\times \Omega_i) $ for each $i$, and concludes the proof of the theorem.
\end{proof}
\begin{remark}
  In the case $\div \av =0$, if $p_{j,i}-p_{i,j}-\av_i \cdot \nv_i=0$ and $c_i \ge \alpha_0 >0$,
  then $C_1=0$ in \eqref{eq:estcontal} and we conclude without using Gr\"{o}nwall's lemma.
\end{remark}

\subsection{Order 2 transmission conditions}
\leavevmode\par
\noindent\begin{theorem}\label{th:convcontorder2}
  Assume $p_{i,j} \in W^{1,\infty}(\Omega_i), \ p_{i,j}+p_{j,i}>0 \ a.e.$,
  $q_{i,j}=q > 0$, $\av_i\in (W^{1,\infty}(\Omega_i))^N$, $\nu_i\in
  W^{1,\infty}(\Omega)$, $\rv_{i,j}\in
  (W^{1,\infty}(\Omega_i))^{N-1}$, with $\rv_{i,j}=\rv_{j,i} \ on
  \ \Gamma_{i,j}$, $s_{i,j}\in W^{1,\infty}(\Omega_i)$, $s_{i,j}>0$
with $s_{i,j}=s_{j,i} \ on \ \Gamma_{i,j}$, and the domain is cut in bands as in Figure \ref{fig:decompb}, right. Then, the algorithm \eqref{eq:algoorder2} converges in each subdomain to the solution $u$ of problem \eqref{eq:adall}.
\end{theorem}
\begin{proof}
We first need some results in differential geometry. For any $i \in \inter{1,I}$,
For every $j \in {\cal N}_i$, the
normal vector $ \nv_{i}$ can be extended in a neighbourhood of
$\Gamma_{i,j}$ in $\Omega_i$ as a smooth function $\tilde{\nv}_{i}$
with length one. Let $\psi_{i,j} \in {\cal
  C}^\infty(\overline{\Omega_i})$, such that $\psi_{i,j} \equiv 1$ in
a neighbourhood of $\Gamma_{i,j}$, $\psi_{i,j} \equiv 0$ in a
neighbourhood of $\Gamma_{i,k}$ for $k \in {\cal N}_i, \, k \neq j$
and $\sum_{j \in {\cal N}_i}\psi_{i,j}>0$ on $\Omega_i$.  We can assume that
$\tilde{\nv}_{i}$ is  defined on a neighbourhood of the support of
$\psi_{i,j}$. We extend the tangential gradient and divergence
operators in the support of $ \psi_{i,j}$ by:
\begin{eqnarray}
\gradts{\Gamma_{i,j}}\phi:= \nabla \phi-(\partial_{\tilde{\nv}_{i}}\phi)\tilde{\nv}_{i},\quad
\divts{\Gamma_{i,j}}\xv:=\div(\xv-(\xv \cdot \tilde{\nv}_{i})\tilde{\nv}_{i}).\nonumber
\end{eqnarray}
It is easy to see that
$(\gradts{\Gamma_{i,j}}\phi)_{|\Gamma_{i,j}}=\grads{\Gamma_{i,j}}\phi$,
$(\divts{\Gamma_{i,j}}\xv)_{|\Gamma_{i,j}}=\divs{\Gamma_{i,j}}\xv$ and
for $\xv$ and $\chi$ with support in $supp(\psi_{i,j})$, we have
\begin{equation}\label{japhet_plenary_eq:divs}
\int_{\Omega_i} (\divts{\Gamma_{i,j}}\xv)\, \chi\, dx
=-\int_{\Omega_i} \xv \cdot \gradts{\Gamma_{i,j}} \chi\, dx.
\end{equation}
Now we prove Theorem \ref{th:convcontorder2}. We consider the
algorithm \eqref{eq:algoorder2} on the error, so we suppose
$f=u_0=0$. We set $\|\varphi\|_i=\|\varphi\|_{L^2(\Omega_i)}$,
$\interleave\varphi\interleave_i^2= \|\sqrt{\nu_i}\,\nabla
\varphi\|^2_i$,
$\|\varphi\|_{i,\infty}=\|\varphi\|_{L^{\infty}(\Omega_i)}$,
$\|\varphi\|_{i,1,\infty}=\|\varphi\|_{W^{1,\infty}(\Omega_i)}$ and
$\beta_{i,j}=\sqrt{\frac{p_{i,j}+p_{j,i}}{2}}$.
\vspace{3mm}

The proof is based on energy estimates containing the term
$$\int_0^t \int_{\Gamma_{i,j}} \left(\nu_i \partial_{\nv_i}u_i^k
-\av_i\cdot\nv_iu_i^k + \mathcal{S}_{i,j}
u_i^k\right)^2\,d\sigma\,d\tau,
$$
and that we derive by multiplying successively the first equation of
\eqref{eq:algoorder2} by the terms $\beta_i^2\,u_i^k$, $\partial_t
u_i^k$, $\divts{\Gamma_{i,j}}(\psi_{i,j}^2\rv_{i,j} u_i^k )$ and
$-\divts{\Gamma_{i,j}} (\psi_{i,j}^2 s_{i,j} \,\gradts{\Gamma_{i,j}}\,
u_i^k)$.

We multiply the first equation of \eqref{eq:algoorder2} by
$\beta_i^2\,u_i^k$, integrate on $(0,t)\times \Omega_i$ and
integrate by parts in space,
 \begin{multline}\label{eq:estcontorder2-1}
  \frac{1}{2}\|\beta_iu_i^k(t)\|^2_i
   +\int_0^t\interleave \beta_iu_i^k(\tau,\cdot)\interleave_i^2 \,d\tau
   -\int_0^t \int_{\Omega_i} \beta_i(\av_i \cdot \nabla \beta_i)(u_i^k)^2 \,dx\,d\tau
    \hspace{1cm}\\
   +\int_0^t \int_{\Omega_i} (c_i+\frac{1}{2}\div\av_i)\beta_i^2(u_i^k)^2 \,dx\,d\tau
   -  \int_0^t \int_{\Omega_i} \nu_i|\nabla \beta_i|^2(u_i^k)^2 \,dx\,d\tau \\
   -\int_0^t \int_{\Gamma_{i,j}} \beta_{i,j}^2(\nu_i\partial_{\nv_i} u_i^k
   - \frac{\av_i\cdot\nv_i}{2}u_i^k)\, u_i^k \, d\sigma\,d\tau
   =0.
 \end{multline}
We multiply the first equation of \eqref{eq:algoorder2} by $\partial_t
u_i^k$, integrate on $(0,t)\times \Omega_i$ and then integrate by
parts in space,
\begin{multline}\label{eq:estcontorder2-2}
  \hspace{-8mm}
   \int_0^t \|\partial_t u_i^k\|^2_{i} \,d\tau
 + \frac{1}{2}\interleave u_i^k(t)\interleave_i^2
 +\int_0^t\int_{\Omega_i}(c_iu_i^k+\div(\av_i u_i^k)) \, \partial_t  u_i^k
 \,dx\,d\tau
 - \int_0^t \int_{\Gamma_{i,j}} \nu_i\partial_{\nv_i} u_i^k
  \, \partial_t u_i^k \, d\sigma\,d\tau
 =0.
\end{multline}
We multiply the first equation of \eqref{eq:algoorder2} by
$\divts{\Gamma_{i,j}} (\psi_{i,j}^2\rv_{i,j} u_i^k )$ integrate on
$(0,t)\times \Omega_i$ and integrate by parts in space:
\begin{multline}\label{eq:estcontorder2-3-1}
  \int_0^t\int_{\Omega_i}\partial_t u_i^k
    \divts{\Gamma_{i,j}}(\psi_{i,j}^2\rv_{i,j} u_i^k)\,dx\,d\tau
   +\int_0^t\int_{\Omega_i}\div(\av_i u_i^k) \,
    \divts{\Gamma_{i,j}}(\psi_{i,j}^2\rv_{i,j} u_i^k)\,dx\,d\tau \\
   +\int_0^t\int_{\Omega_i} c_i u_i^k
       \divts{\Gamma_{i,j}}(\psi_{i,j}^2\rv_{i,j} u_i^k) \,dx\,d\tau\\
   + \int_0^t\int_{\Omega_i} \nu_i \nabla u_i^k \cdot
     \nabla\divts{\Gamma_{i,j}}(\psi_{i,j}^2\rv_{i,j} u_i^k)
      \,dx\,d\tau
  -\int_0^t\int_{\Gamma_{i,j}} \nu_i\partial_{\nv_i} u_i^k
     \, \divs{\Gamma_{i,j}}(\rv_{i,j} u_i^k)  \,d\sigma\,d\tau
   =0.
\end{multline}
We observe that
\begin{multline}\label{eq:estcontorder2-3-2}
 \int_0^t\int_{\Omega_i} \nu_i \nabla u_i^k \cdot
   \nabla\divts{\Gamma_{i,j}}(\psi_{i,j}^2\rv_{i,j} u_i^k)
    \,dx\,d\tau\\
= \int_0^t\int_{\Omega_i} \nu_i \nabla u_i^k \cdot
   \nabla(\divts{\Gamma_{i,j}}(\psi_{i,j}^2\rv_{i,j}) u_i^k)
    \,dx\,d\tau
+\int_0^t\int_{\Omega_i} \nu_i \nabla u_i^k \cdot
   \nabla (\psi_{i,j}^2\rv_{i,j} \cdot \gradts{\Gamma_{i,j}}\, u_i^k)
    \,dx\,d\tau,
\end{multline}
with
\begin{multline}\label{eq:estcontorder2-3-3}
  \int_0^t\int_{\Omega_i} \nu_i \nabla u_i^k \cdot
    \nabla (\psi_{i,j}^2\rv_{i,j} \cdot \gradts{\Gamma_{i,j}}\, u_i^k)
     \,dx\,d\tau\\
  \ge -{1 \over 4}\int_0^t \| \psi_{i,j}\sqrt{\nu_i\, s_{i,j}}\
  \nabla \gradts{\Gamma_{i,j}}\, u_i^k \|_i^2 \,d\tau
  -C \int_0^t\int_{\Omega_i} (\|\nabla u_i^k\|_i^2+\|u_i^k\|_i^2)
  \,dx\,d\tau.
\end{multline}
Replacing \eqref{eq:estcontorder2-3-3} in \eqref{eq:estcontorder2-3-2}
and then \eqref{eq:estcontorder2-3-2} in \eqref{eq:estcontorder2-3-1},
we obtain
\begin{equation}\label{eq:estcontorder2-3}
\begin{array}{l}
 \hspace{-4mm}{\small \mbox{$ \ds\int_0^t\int_{\Omega_i}\left(
  \partial_t u_i^k \divts{\Gamma_{i,j}}(\psi_{i,j}^2\rv_{i,j} u_i^k)
  + \div(\av_i u_i^k) \,
    \divts{\Gamma_{i,j}}(\psi_{i,j}^2\rv_{i,j} u_i^k)
    + c_i u_i^k
    \divts{\Gamma_{i,j}}(\psi_{i,j}^2\rv_{i,j} u_i^k) \right)\,dx\,d\tau$}} \\
  \ds\hspace{15mm} -{1 \over 4}\int_0^t \| \psi_{i,j}\sqrt{\nu_i\, s_{i,j}}\
  \nabla \gradts{\Gamma_{i,j}}\, u_i^k )\|_i^2 \,d\tau
  -\int_0^t\int_{\Gamma_{i,j}} \nu_i\partial_{\nv_i} u_i^k
   \, \divs{\Gamma_{i,j}}(\rv_{i,j} u_i^k)  \,d\sigma\,d\tau\\
  \ds\hspace{30mm}\le C \int_0^t\int_{\Omega_i} (\|\sqrt{\nu_i}\nabla u_i^k\|_i^2+\|\beta_iu_i^k\|_i^2)
  \,dx\,d\tau.
  \end{array}
\end{equation}
Now we multiply the first equation of \eqref{eq:algoorder2} by
$-\divts{\Gamma_{i,j}} (\psi_{i,j}^2 s_{i,j} \,\gradts{\Gamma_{i,j}}\,
u_i^k)$ integrate on $(0,t)\times \Omega_i$ and integrate by parts in
space:
\begin{multline}\label{eq:estconto2-4}
  \frac{1}{2} \| \psi_{i,j}\sqrt{s_{i,j}} \, \gradts{\Gamma_{i,j}} \,u_i^k (t)\|_i^2
   -\int_0^t\int_{\Omega_i}\div(\av_i u_i^k) \, \divts{\Gamma_{i,j}}
      (\psi_{i,j}^2 s_{i,j} \, \gradts{\Gamma_{i,j}}\, u_i^k)\,dx\,d\tau\\
    \hspace{-4mm}{\small \mbox{$ \ds+\int_0^t\int_{\Omega_i} \psi_{i,j}^2 s_{i,j} \gradts{\Gamma_{i,j}}( c_i u_i^k )
     \cdot\gradts{\Gamma_{i,j}}\, u_i^k \,dx\,d\tau
   - \int_0^t\int_{\Omega_i} \nu_i \nabla u_i^k \cdot
    \nabla (\divts{\Gamma_{i,j}}
     (\psi_{i,j}^2 s_{i,j} \, \gradts{\Gamma_{i,j}}\, u_i^k))
    \,dx\,d\tau$}}\\
   + \int_0^t\int_{\Gamma_{i,j}} \nu_i\partial_{\nv_i} u_i^k
    \, \divs{\Gamma_{i,j}}(s_{i,j} \,
    \grads{\Gamma_{i,j}}\, u_i^k)  \,d\sigma\,d\tau
   =0.
\end{multline}
We have,
\begin{align}\label{eq:estgradgrad}
  - \int_0^t\int_{\Omega_i} &\nu_i \nabla u_i^k \cdot
   \nabla (\divts{\Gamma_{i,j}}
       (\psi_{i,j}^2 s_{i,j} \, \gradts{\Gamma_{i,j}}\, u_i^k))
   \,dx\,d\tau \nonumber\\
  &\ge
    {1 \over 2}\int_0^t \| \psi_{i,j}\sqrt{\nu_{i}\, s_{i,j}}\
  \nabla \gradts{\Gamma_{i,j}}\, u_i^k )\|_i^2 \,d\tau
  - C_1 \int_0^t \|\nabla u_i^k \|_i^2 \,d\tau.
\end{align}
%
Replacing \eqref{eq:estgradgrad} in \eqref{eq:estconto2-4} leads to
\begin{multline}\label{eq:estcontorder2-4}
  \frac{1}{2} \| \psi_{i,j}\sqrt{s_{i,j}} \, \gradts{\Gamma_{i,j}} \,u_i^k (t)\|_i^2
  +{1 \over 2}\int_0^t \|\psi_{i,j}\sqrt{\nu_{i}\, s_{i,j}}\
  \nabla \gradts{\Gamma_{i,j}}\, u_i^k )\|_i^2 \,d\tau\\[-1mm]
    +\int_0^t\int_{\Omega_i}\psi_{i,j}^2 s_{i,j}\,c_i |\gradts{\Gamma_{i,j}}u_i^k|^2\,dx\,d\tau
  + \int_0^t\int_{\Gamma_{i,j}} \nu_i \partial_{\nv_i}u_i^k
     \, \divs{\Gamma_{i,j}}( s_{i,j} \,
    \grads{\Gamma_{i,j}}\, u_i^k)  \,d\sigma\,d\tau\\[-1mm]
\hspace{-0.5cm}
    \le
   \int_0^t\int_{\Omega_i}\div(\av_i u_i^k) \,
    \divts{\Gamma_{i,j}}
    (\psi_{i,j}^2  s_{i,j} \, \gradts{\Gamma_{i,j}}\, u_i^k)\,dx\,d\tau
  +C\int_0^t \|\sqrt{\nu_i}\nabla u_i^k\|_i^2\,d\tau.
\end{multline}
Multiplying \eqref{eq:estcontorder2-2}, \eqref{eq:estcontorder2-3} and
\eqref{eq:estcontorder2-4} by $q$, and adding the three equations with
\eqref{eq:estcontorder2-1}, we get
 \begin{multline*}
   \frac{1}{2} \left(
     \|\beta_iu_i^k(t)\|^2_i
     + q\interleave u_i^k(t)\interleave_i^2
     +q \| \psi_{i,j}\sqrt{s_{i,j}} \, \gradts{\Gamma_{i,j}} \,u_i^k (t)\|_i^2
     \right)
 + \int_0^t\interleave \beta_i u_i^k(\tau,\cdot)\interleave_i^2 \,d\tau\\
 +q \int_0^t \|\partial_t u_i^k\|^2_{i} \,d\tau
 +\frac{q}{4}
    \int_0^t \| \psi_{i,j}\sqrt{\nu_i \,s_{i,j}}\ \nabla
    \gradts{\Gamma_{i,j}}\, u_i^k \|_i^2 \,d\tau\\
 -\int_0^t \int_{\Gamma_{i,j}} \beta_i^2(\nu_i\partial_{\nv_i} u_i^k
   - \frac{\av_i\cdot\nv_i}{2}u_i^k)\, u_i^k
  \, d\sigma\,d\tau\\
 -q
    \int_0^t \int_{\Gamma_{i,j}} \nu_i\partial_{\nv_i} u_i^k
  \, \left(\partial_t u_i^k
  + \divs{\Gamma_{i,j}}(\rv_{i,j} u_i^k)
  - \divs{\Gamma_{i,j}}(s_{i,j} \, \grads{\Gamma_{i,j}}\, u_i^k)\right)
    \,d\sigma\,d\tau\\
\le
 q(\frac{1}{2}\|\av_i\|_{i,1,\infty}
  + \|\rv_{i,j}\|_{i,1,\infty})
 \int_0^t \|u_i^k\|_i \, \|\partial_t u_i^k \|_i \,d\tau\\
 + q(\|\av_i\|_{i,\infty}
    +\|\rv_{i,j}\|_{i,\infty})
  \int_0^t \|\nabla u_i^k\|_i \, \|\partial_t u_i^k \|_i
\,d\tau\\
 + \frac{q}{2} \|\av_i\|_{i,\infty}
   \int_0^t \|\nabla u_i^k\|_i \,
     \|\divts{\Gamma_{i,j}}
 (\psi_{i,j}^2 s_{i,j} \, \gradts{\Gamma_{i,j}}\, u_i^k)\|_i\,d\tau \\
 +  q(\frac{1}{2}\|\av_i\|_{i,1,\infty} +\|c_i\|_{i,\infty})
   \int_0^t \| u_i^k\|_i \,
     \|\divts{\Gamma_{i,j}}
 (\psi_{i,j}^2 s_{i,j} \, \gradts{\Gamma_{i,j}}\, u_i^k)\|_i\,d\tau \\
     +\frac{q}{2}(\|\av_i\|_{i,1,\infty} +\|c_i\|_{i,\infty})\|u_i^k(t)\|_i^2
 + C \left(\int_0^t \|\beta_iu_i^k\|_i^2 \,d\tau
    +q \int_0^t \|\sqrt{\nu_i}\nabla u_i^k\|_i^2 \,d\tau \right).
\end{multline*}
We bound the right-hand side by
\begin{multline*}
 {1 \over 2} \left(
 \frac{q}{2} \int_0^t \|\partial_t u_i^k \|_i^2 \,d\tau
 +\frac{q}{4} \int_0^t
     \|\psi_{i,j}\sqrt{\nu_i \, s_{i,j}} \ \nabla \gradts{\Gamma_{i,j}}\, u_i^k\|_i^2\,d\tau
  \right)\\
   +\frac{q}{2}(\|\av_i\|_{i,1,\infty} +\|c_i\|_{i,\infty})\|u_i^k(t)\|_i^2
 + C \left(\int_0^t \|\beta_iu_i^k\|_i^2 \,d\tau
    +q \int_0^t \|\sqrt{\nu_i}\nabla u_i^k\|_i^2 \,d\tau \right).
\end{multline*}
We simplify the terms which appear on both sides, and obtain
\begin{multline}\label{eq:estbtm4}
 \frac{1}{2} \left(
   \|\beta_iu_i^k(t)\|^2_i
   + q\interleave u_i^k(t)\interleave_i^2
   +q \| \psi_{i,j}\sqrt{s_{i,j}} \, \gradts{\Gamma_{i,j}} \,u_i^k (t)\|_i^2
   \right)
 + \int_0^t\interleave \beta_iu_i^k(\tau,\cdot)\interleave_i^2 \,d\tau\\
 +\frac{q}{2} \int_0^t \|\partial_t u_i^k\|^2_{i} \,d\tau
 +\frac{q}{8}
    \int_0^t \|\psi_{i,j} \sqrt{\nu_i \, s_{i,j}}\ \nabla
    \gradts{\Gamma_{i,j}}\, u_i^k \|_i^2 \,d\tau\\
 -\int_0^t \int_{\Gamma_{i,j}} \beta_i^2(\nu_i\partial_{\nv_i} u_i^k
   - \frac{\av_i\cdot\nv_i}{2}u_i^k)\, u_i^k
  \, d\sigma\,d\tau\\
 -q
    \int_0^t \int_{\Gamma_{i,j}} \nu_i\partial_{\nv_i} u_i^k
  \, \left(\partial_t u_i^k
+ \divs{\Gamma_{i,j}}(\rv_{i,j} u_i^k)
  - \divs{\Gamma_{i,j}}(s_{i,j} \, \grads{\Gamma_{i,j}}\, u_i^k)\right)
    \,d\sigma\,d\tau\\
\le
C \left(\int_0^t \|\beta_iu_i^k\|_i^2 \,d\tau
    +q \int_0^t \|\sqrt{\nu_i}\nabla u_i^k\|_i^2 \,d\tau \right).
\end{multline}
%
Recalling that $s_{i,j}=s_{j,i}$ on $\Gamma_{i,j}$ and
$\rv_{i,j}=\rv_{j,i}$ on $\Gamma_{i,j}$, we use now:
\begin{multline}\label{eq:esttransmcd}
  \left(\nu_i \partial_{\nv_i} u_i^k
   - \av_i\cdot\nv_i u_i^k + \mathcal{S}_{i,j} u_i^k\right)^2
- \left(\nu_i \partial_{\nv_i} u_i^k
   - \av_i\cdot\nv_i u_i^k - \mathcal{S}_{j,i} u_i^k\right)^2\\
\hspace{-4mm}{\small \mbox{$ \ds = 4\left(\beta_{i,j}^2( \nu_i \partial_{\nv_i} u_i^k
   - \frac{\av_i\cdot\nv_i}{2}u_i^k)u_i^k + q \nu_i \partial_{\nv_i} u_i^k
   (\partial_t u_i^k
+\divs{\Gamma_{i,j}}(\rv_{i,j} u_i^k)
-\divs{\Gamma_{i,j}}(s_{i,j} \, \grads{\Gamma_{i,j}}\, u_i^k))\right)$}}
\\
    +2q(p_{i,j}-p_{j,i}-2\av_i\cdot\nv_i)(\partial_t u_i^k
+\divs{\Gamma_{i,j}}(\rv_{i,j} u_i^k)
-\divs{\Gamma_{i,j}}(s_{i,j} \, \grads{\Gamma_{i,j}}\, u_i^k))u_i^k\\
+(p_{i,j}+p_{j,i})(p_{i,j}-p_{j,i}-\av_i\cdot\nv_i)(u_i^k)^2.
\end{multline}
Replacing \eqref{eq:esttransmcd} into \eqref{eq:estbtm4}, we
obtain
\vspace{-1.5mm}
\begin{multline}\label{eq:esttransmcd2}
  {1\over 2}\left(
   \|\beta_{i} u_i^k(t)\|^2_i
   + q\interleave u_i^k(t)\interleave_i^2
   +q \|\psi_{i,j}\sqrt{s_{i,j}} \, \gradts{\Gamma_{i,j}} \,u_i^k (t)\|_i^2
   \right)
 + \int_0^t\interleave \beta_{i}u_i^k(\tau,\cdot)\interleave_i^2 \,d\tau \\
 +\frac{q}{2} \int_0^t \|\partial_t u_i^k\|^2_{i} \,d\tau
 +\frac{1}{4}
   \int_0^t \int_{\Gamma_{i,j}}
    \left(\nu_i \partial_{\nv_i}u_i^k
  -\av_i\cdot\nv_i \, u_i^k - \mathcal{S}_{j,i} u_i^k\right)^2\,d\sigma\,d\tau \\
  +\frac{q}{8}
    \int_0^t \| \psi_{i,j}\sqrt{\nu_i \, s_{i,j}}\ \nabla
    \gradts{\Gamma_{i,j}}\, u_i^k \|_i^2 \,d\tau
  \le
\frac{1}{4}
   \int_0^t \int_{\Gamma_{i,j}}
    \left(\nu_i \partial_{\nv_i}u_i^k
  -\av_i\cdot\nv_i \, u_i^k + \mathcal{S}_{i,j} u_i^k\right)^2\,d\sigma\,d\tau \\
    + \int_0^t \int_{\Gamma_{i,j}}
 (p_{i,j}+p_{j,i})(-p_{i,j}+p_{j,i}+\av_i\cdot\nv_i)(u_i^k)^2 \,d\sigma\,d\tau
+\frac{q}{2}(\|\av_i\|_{i,1,\infty} +\|c_i\|_{i,\infty})\|u_i^k(t)\|_i^2   \\
+{q \over 2}\int_0^t \int_{\Gamma_{i,j}}
 (-p_{i,j}+p_{j,i}+2\av_i\cdot\nv_i)
(\partial_t u_i^k +\divs{\Gamma_{i,j}}(\rv_{i,j} u_i^k)
-\divs{\Gamma_{i,j}}(s_{i,j} \, \grads{\Gamma_{i,j}}\, u_i^k))\,u_i^k\,d\sigma\,d\tau\\
+ C \left(\int_0^t \|\beta_{i}u_i^k\|_i^2 \,d\tau
    +q \int_0^t \|\sqrt{\nu_i}\nabla u_i^k\|_i^2 \,d\tau \right).
\end{multline}
In order to estimate the fourth term in the right-hand side of
\eqref{eq:esttransmcd2}, we observe that
\begin{equation*}
\int_0^t \int_{\Gamma_{i,j}} (-p_{i,j}+p_{j,i}+2\av_i\cdot\nv_i)u_i^k
\,\partial_t u_i^k\,d\sigma\,d\tau
= {1 \over 2} \int_{\Gamma_{i,j}} (-p_{i,j}+p_{j,i}+2\av_i\cdot\nv_i)u_i^k (t)^2 \,d\sigma.
\end{equation*}
By the trace theorem in the right-hand side, we write:
\begin{equation*}
\int_0^t \int_{\Gamma_{i,j}} (-p_{i,j}+p_{j,i}+2\av_i\cdot\nv_i)u_i^k
\,\partial_t u_i^k\,d\sigma\,d\tau
\le C \|u_i^k (t)\|_i \|\sqrt{\nu_i}\nabla u_i^k (t)\|_i,
\end{equation*}
and
\vspace{-3mm}
\begin{equation}\label{eq:estut}
\|u_i^k (t)\|_i^2 = 2 \int_0^t \int_{\Omega_i} (\partial_t u_i^k)u_i^k
\le 2 \left(\int_0^t \|\partial_t u_i^k\|_i^2\right)^{1 \over 2}
      \left(\int_0^t \|u_i^k\|_i^2\right)^{1 \over 2},
\end{equation}
we obtain
\begin{multline}\label{eq:esttrace}
{q \over 2}\
\int_0^t \int_{\Gamma_{i,j}} (-p_{i,j}+p_{j,i}+2\av_i\cdot\nv_i) u_i^k
\,\partial_t u_i^k\,d\sigma\,d\tau\\
\le\frac{q}{8} \int_0^t \| \partial_t u_i^k\|_i^2 \,d\tau
+\frac{q}{4}\interleave u_i^k(t)\interleave_i^2
+C \left(\int_0^t  \| \beta_{i}u_i^k\|_i^2  \,d\tau
\right).
\end{multline}
\vspace{-1mm}
Moreover, integrating by parts and using the trace theorem, we have
\vspace{-1mm}
\begin{multline}\label{eq:estimtt1}
\ds  -\frac{q}{2} \int_0^t \int_{\Gamma_{i,j}}
     \divs{\Gamma_{i,j}}(s_{i,j} \, \grads{\Gamma_{i,j}}\, u_i^k)
    (-p_{i,j}+p_{j,i}+2\av_i\cdot\nv_i) u_i^k \,d\sigma\,d\tau \\
     \hspace{-1.5cm}
\ds \le \frac{q}{16} \int_0^t
         \|\psi_{i,j}\sqrt{\nu_i \, s_{i,j}}\ \nabla \gradts{\Gamma_{i,j}}\, u_i^k \|_i^2 \,d\tau
+ C (\int_0^t \| \gradts{\Gamma_{i,j}}\, u_i^k \|_i^2\,d\tau
+ \int_0^t \| \beta_{i} u_i^k \|_i^2\,d\tau).
\end{multline}
Using \eqref{eq:estut}, we estimate the third term in the right-hand
side of \eqref{eq:esttransmcd2} by
\begin{equation}\label{eq:esttrace2}
\frac{q}{2}(\|\av_i\|_{i,1,\infty} +\|c_i\|_{i,\infty})\|u_i^k(t)\|_i^2
\le\frac{q}{8} \int_0^t \| \partial_t u_i^k\|_i^2 \,d\tau
+C \int_0^t  \| \beta_{i}u_i^k\|_i^2  \,d\tau.
\end{equation}
Replacing \eqref{eq:estimtt1}, \eqref{eq:esttrace} and
\eqref{eq:esttrace2} in \eqref{eq:esttransmcd2}, then using
the transmission conditions, we have:
\begin{multline*}
  {1\over 2}\left(
   \|\beta_{i}  u_i^k(t)\|^2_i
   + \frac{q}{2}\interleave u_i^k(t)\interleave_i^2
   +q \| \psi_{i,j}\sqrt{s_{i,j}} \, \gradts{\Gamma_{i,j}} \,u_i^k (t)\|_i^2
   \right)\\
  + \int_0^t\interleave \beta_{i} u_i^k(\tau,\cdot)\interleave_i^2 \,d\tau
  +\frac{q}{4} \int_0^t \|\partial_t u_i^k\|^2_{i} \,d\tau
 +\frac{q}{16}
    \int_0^t \| \psi_{i,j}\sqrt{\nu_i\, s_{i,j}}\ \nabla \,
    \gradts{\Gamma_{i,j}}\, u_i^k \|_i^2 \,d\tau\\
 +\frac{1}{4}
   \int_0^t \int_{\Gamma_{i,j}}
    \left(\nu_i \partial_{\nv_i}u_i^k
  -\av_i\cdot\nv_i \, u_i^k - \mathcal{S}_{j,i} u_i^k\right)^2\,d\sigma\,d\tau\\
 \le
 \frac{1}{4}
   \int_0^t \int_{\Gamma_{i,j}}
    \left(\nu_j \partial_{\nv_i}u_j^{k-1}
  -\av_j\cdot\nv_i \, u_j^{k-1} +\mathcal{S}_{i,j}
     u_j^{k-1}\right)^2\,d\sigma\,d\tau\\
 + C \left(\int_0^t \|\beta_{i} u_i^k\|_i^2 \,d\tau
    +\frac{q}{2}  \int_0^t \|\sqrt{\nu_i}\,\nabla u_i^k\|_i^2 \,d\tau \right).
\end{multline*}
We now sum up over the interfaces $j \in {\cal N}_i$, then over the
subdomains for $1 \le i \le I$, and on the iterations for $1 \le k \le
K$, the boundary terms cancel out, and we obtain for any $t \in
(0,T)$,
\begin{multline}
\sum_{k\in \inter{1,K}}\sum_{i\in \inter{1,I}}
  \left(\|\beta_{i,j}u_i^k(t)\|^2_i+q \|\sqrt{\nu_i}\,\nabla u_i^k(t)\|_i^2
+\nu_0 \int_0^t \|\nabla (\beta_{i,j}u_i^k) \|_i^2 \,d\tau \right)\\
\le \alpha(t)+C \sum_{k\in \inter{1,K}}\sum_{i\in \inter{1,I}}
\left(\int_0^t \|\beta_{i,j}u_i^k\|_i^2 \,d\tau
    +q \int_0^t \|\sqrt{\nu_i}\,\nabla u_i^k\|_i^2 \,d\tau \right),
\end{multline}
with
\begin{equation}
\alpha(t)=\frac{1}{4}\sum_{i\in \inter{1,I}} \sum_{j\in{\cal N}_i}
 \int_0^t \int_{\Gamma_{i,j}}
  \left(\nu_j \partial_{\nv_i} u_i^0
  -\av_j\cdot\nv_i \, u_j^0 +\mathcal{S}_{i,j}
     u_j^0\right)^2\,d\sigma\,d\tau.
\end{equation}
%
We conclude with Gr\"{o}nwall's lemma as before.
\end{proof}
%
\section{The discontinous Galerkin time stepping for the Schwarz waveform relaxation algorithm}
\label{section:dgoswr}
%
In the following sections, in order to simplify the analysis, we suppose that $c+{1 \over 2}\div \av \ge \alpha_0> 0$ \textit{
  a.e.} in $\Omega$.\\
%

\subsection{Time discretization of the local problem: discontinuous Galerkin method}

We suppose that the coefficients are
restricted to $p -\frac{\av\cdot\nv}{2}+{q \over 2}\nabla_{\Gamma}
\cdot\rv >0$ \ \textit{a.e.} on $\Gamma$, $q \ge 0 \ a.e.$ and $s>0$ \textit{a.e.}. This implies that the bilinear form $a$  defined in \eqref{eq:formbilin}
%
is positive definite on $H^1(\co)$ when $q=0$, and
positive definite on $H^1_1(\co)$ when $q\ge q_0 >0$ a.e.

We recall the time-discontinuous Galerkin method, as presented in
\cite{johnson:1985:dgfirst}.  We are given a decomposition ${\cal T}$
of the time interval $(0,T)$, $I_n=(t_n,t_{n+1}]$, for $0 \le n \le N$,
the mesh size is $k_n=t_{n+1}-t_n$. For ${\cal B}$ a Banach space and
$I$ an interval of $\R$, define for any integer $d \ge 0$
\begin{eqnarray}
  \begin{array}{rcl}
    {\mathbf{P}}_{d}({\cal B})&=&\{\varphi:I\rightarrow {\cal B}, \ \varphi(t)=
    \displaystyle\sum_{i=0}^{d } \varphi_i t^i,
    \ \varphi_i \in {\cal B}\} , \\
    {\cal P}_d({\cal B},{\cal T})&=&\{\varphi:I\rightarrow {\cal B}, \
    \varphi_{\mid I_{n}} \in
    {\mathbf{P}}_{d}({\cal B}), \ 1 \le n \le N\}.
  \end{array} \nonumber
\end{eqnarray}
Let ${\cal B}=H_1^1(\co)$ if $q > 0$, ${\cal B}=H^1(\co )$ if $q = 0$.
We define an approximation $U$ of $u$, polynomial of degree lower than
$d$ on every subinterval $I_n$. For every point $t_n$, we define
$U(t_n^{-})=\lim_{t \rightarrow t_n-0} U(t)$, and note
$U(t_n^{+})=\lim_{t \rightarrow t_n+0} \uud(t)$. The time discretization of \eqref{eq:formvariat} leads to searching $U\in
{\cal P}_d({\cal B},{\cal T})$ such that
\begin{eqnarray}\label{eq:dGvariat}
  \left\{ \begin{array}{l}
      U(0,\cdot)=u_0, \\
      \forall \, V \in {\cal P}_d({\cal B},{\cal T})\colon
      \displaystyle
      \int_{I_{n}} (\,m(\dot{U},V ) +a(U,V))\,dt\\
      \hspace{20mm}+m(U(t_n^+,\cdot)-U(t_n^-,\cdot),V(t_n^+,\cdot))
      = \displaystyle\int_{I_{n}}L(V ) \,dt,
      \quad
    \end{array}\right.
\end{eqnarray}
with $L(V )= (f,V)_{L^2(\co)}+(g,V)_{L^2(\Gamma)}$.  Since $I_n$ is
closed at $t_{n+1}$, $U(t_{n+1}^-)$ is the value of $U$ at $t_{n+1}$.
Due to the discontinuous nature of the test and trial spaces, the
method is an implicit time stepping scheme, and $U\in
{\mathbf{P}}_d({\cal B},{\cal T})$ is obtained recursively on each
subinterval, which makes it very flexible.
\begin{theorem}\label{th:existdisc}
  If $p -\frac{\av\cdot\nv}{2}+{q \over 2}\nabla_{\Gamma} \cdot\rv >0$
  \ \textit{a.e.} on $\Gamma$, $q \ge 0 \ a.e.$ and $s>0 \ a.e.$,
  equation \eqref{eq:dGvariat} defines a unique solution.
\end{theorem}
\begin{proof}
The result relies on the fact that the bilinear form $a$ is definite
positive.  It is is most easily seen by using a basis of Legendre
polynomials.  $U\in {\cal P}_d(H_1^1(\Omega),{\cal T})$ is obtained
recursively on each subinterval. We introduce the Legendre polynomials
$L_n$, orthogonal basis in $L^2(-1,1)$, with $L_n(1)=1$. $L_n$ has the
parity of $n$, hence $L_n(-1)=(-1)^n$. A basis of orthogonal
polynomial on $I_n$ is given by $L_{n,k}(t)=L_k(\frac{2}{k_n}
(t-\frac{t_{n+1}+t_n}{2}))$.  Choose $V(t,x)=L_{n,j}(t)\Phi_j(x)$ in
\eqref{eq:dGvariat} with $\Phi_j\in H_1^1(\Omega)$, and expand $U$ on
$I_n$ as $U(t,x)=\sum_{k=0}^d U_k(x)L_{n,k}(t)$. Suppose $U$ to be
given on $(0,t_n]$. In order to determine $U$ on $I_n$, we must solve
  the system: for any $\Phi_j\in H_1^1(\Omega)$,
\[
  \begin{split}
    \sum_{k=0}^d\ \int_{I_{n}} \biggl(\dot{L}_{n,k}L_{n,j}m(U_k,\Phi_j) +
    L_{n,k}L_{n,j}a(U_k,\Phi_j) \biggr) \,dt\\
    + \sum_{k=0}^d L_{n,k}(t_n^+)L_{n,j}(t_n^+)m(U_k,\Phi_j)=
    \displaystyle\int_{I_{n}} L_{n,j} L(\Phi_j) \,dt.
  \end{split}
\]
It is an implicit scheme. We calculate the coefficients
\begin{align}
  &\int_{I_{n}} L_{n,k}L_{n,j}=\delta_{kj}\|L_{n,j}\|^2,\nonumber\\
  &\int_{I_{n}}\dot{L}_{n,k}L_{n,j}=
  \begin{cases}
    0 & \text{ if $k \le j$}\\
    1-(-1)^{k+j} & \text{ if $k > j$},
  \end{cases} \nonumber\\
  & \int_{I_{n}}\dot{L}_{n,k}L_{n,j}+ L_{n,k}(t_n^+)L_{n,j}(t_n^+)=
  \begin{cases}
    (-1)^{k+j} & \text{ if $k \le j$}\\
    1 & \text{ if $k > j$},
  \end{cases}
  \nonumber
\end{align}
which leads to
{\small
\begin{eqnarray}
 \|L_{n,j}\|^2 a(U_j, \Phi_j)+ m(U_j,\Phi_j)+ \sum_{k<j}(-1)^{k+j}
  m(U_k, \Phi_j)+
  \sum_{k>j}m(U_k, \Phi_j)= \nonumber
  \displaystyle\int_{I_{n}} L_{n,j} L(\Phi_j) \,dt . \nonumber
\end{eqnarray}
}
%
It is a square system of partial differential equations, of the type coercive $+$ compact.  By the Fredholm alternative, we only need to
prove uniqueness. Choose now $\Phi_j=U_j$, and obtain
\[
  \sum_j\|L_{n,j}\|^2a(U_j,U_j)+\sum_jm(U_j,U_j)
  +2\sum_j \sum_{\substack{k>j \\ k+j \ \mbox{even}}}m(U_k, U_j)= 0,
\]
and since $a$ is positive definite, we deduce that $U=0$.
\end{proof}

We will make use of the following remark
(\cite{makridakis:2007:pea}). We introduce the Gauss-Radau points,
$(0<\tau_1,\dotsc,\tau_{d+1}=1)$, defined such that the quadrature
formula
\[
  \int_0^1 f(t) dt\approx \sum_{j=1}^{d+1} w_q \ f(\tau_q)
\]
is exact in $\PP_{2d}$, and the interpolation operator ${\cal I}_n$ on
$[t_n,t_{n+1}]$ at points
$(t_n,t_n+\tau_1k_n,\dotsc,t_n+\tau_{d+1}k_n)$.  For any $\chi \in
\PP_{d}$, $\hat{\chi}={\cal I}_n\chi \in\PP_{d+1}$.

Let ${\cal I}: {\cal P}_d({\cal B},{\cal T}) \rightarrow {\cal
  P}_{d+1}({\cal B},{\cal T})$ be the operator whose restriction to
each subinterval is ${\cal I}_n$ and satisfies ${\cal
  I}U(t_n^+)=U(t_n^-)$.  By using the Gauss-Radau formula, which is
exact in $\PP_{2d}$, we have for all $\psi_{i,j} \in \PP_{d}$
\[
  \int_{I_n>}\dtq{{\cal I\chi}}\psi_{i,j} \, dt- \int_{I_n}\dtq{\chi}\psi_{i,j} \, dt=
  (\chi(t_{n}^+)-\chi(t_{n}^-)) \psi_{i,j}(t_{n}^+).
\]
As a consequence, we have a very useful inequality:
\begin{equation}\label{eq:interpradau2}
  \int_{I_n}\frac{d}{dt}({\cal I}\psi_{i,j})\psi_{i,j} dt
  \ge {1 \over 2}[\psi_{i,j}(t^-_{n+1})^2-\psi_{i,j}(t^-_{n})^2].
\end{equation}
Also, equation \eqref{eq:dGvariat} can be rewritten as
\begin{equation}\label{eq:dGvariatinterp}
  \int_{I_{n}} (\,m(\,\dtq{{\cal I} U },V ) +a(U,V))\,dt=
  \displaystyle\int_{I_{n}}L(V ) \,dt,
\end{equation}
or in  the strong formulation:
\begin{equation}\label{eq:dGstronginterp}
  \begin{array}{l}
    \partial_t ({\cal I} U ) +\div(\av U-\nu\nabla U)+cU= P f, \mbox{ in }
    \Omega\times (0,T) ,\\[2mm]
    \ds\bigl(\nu\,\partial_{\nv}- \av\cdot\nv \bigr)U +
    p\,U + q(\partial_t ({\cal I} U)   + \divs{\Gamma}(\rv U -s\grads{\Gamma}U))
    =Pg \mbox{ on } \Gamma\times (0,T). \\
  \end{array}
\end{equation}
Here $P$ is the projection $L^2$ in each subinterval of ${\cal T}$ on $\PP_d$.
\begin{theorem}[Thomee, \cite{thomee:1997:FEM}]
Let $\uud$ be the solution of \eqref{eq:dGvariat} and $\uu$ the solution of \eqref{eq:pblim}. Under the assumptions of Theorem \ref{th:exfortefo},  the estimate holds
\begin{align}
  \|\uu-\uud\|_{L^\infty(I_n,L^2(\Omega))}\le C k^{d+1}
  \|\partial_t^{d+1}\uu\|_{L^2(0,T;H_2^2(\Omega))},
\end{align}
with $k=\max_{0 \le n \le N} k_n$.
\end{theorem}

\subsection{The discrete in time optimized Schwarz waveform relaxation
algorithm with different subdomains grids}

In this part we present and analyse the discrete non conforming in
time optimized Schwarz waveform relaxation method.

The time partition in subdomain $\Omega_i$, is ${\cal T}_i$, with $N_i
+1$ intervals $I_n^i$, and mesh size $k_n^i$. In view of formulation
\eqref{eq:dGstronginterp}, we define interpolation operators ${\cal
  I}^i$ and projection operators $P^i$ in each subdomain,
i.e. $P^i$ is the projection $L^2$ in each subinterval of ${\cal T}_i$ on $\PP_d$,
and we solve
\begin{subequations}\label{eq:algosdorder2}
  \begin{align}
    & \partial_t ({\cal I}^iU_i^k)
     +\div(\av_i U_i^k -\nu_i\nabla U_i^k)
    +c_i\,U_i^k=P^if\mbox{ in }\Omega_i\times(0,T),\label{eq:algosdorder2-1}\\
    &\ds\bigl(\nu_i \partial_{\nv_i}
    -\av_i\cdot\nv_i \bigr)\,U_i^{k} +S_{i,j}U_i^{k} =
    P^i\bigl((\nu_{j} \partial_{\nv_i}
    -\av_{j}\cdot\nv_{i})\, U_{j}^{k-1}
    +\widetilde{S}_{i,j}U_{j}^{k-1}\bigr)
    \mbox{ on }\Gamma_{i,j}, j\in {\cal N}_i.\label{eq:algosdorder2-2}
  \end{align}
\end{subequations}
Here the operators are different on either part of the "equal" sign:
\begin{equation}
  \begin{array}{l}
    S_{i,j}U= p_{i,j}\,U + q_{i,j}\, (\partial_t ({\cal I}^i U)
    + \divs{\Gamma_{i,j}}(\rv_{i,j}U-s_{i,j}\grads{\Gamma_{i,j}}U)) \\
    \widetilde{S}_{i,j}U= p_{i,j}\,U + q_{i,j}\, (\partial_t ({\cal I}^j U)
    + \divs{\Gamma_{i,j}}(\rv_{i,j}U-s_{i,j}\grads{\Gamma_{i,j}}U)).
  \end{array}
\end{equation}
Formally, the sequence of problems \eqref{eq:algosdorder2} converges to the solution of
\begin{subequations}\label{eq:dgdiscreteall}
  \begin{align}
    & \partial_t ({\cal I}^iU_i)
     +\div(\av_i U_i -\nu_i\nabla U_i)
    +c_i\,U_i=P^if\mbox{ in }\Omega_i\times(0,T),\label{eq:dgdiscreteall-1}\\
    &\ds\bigl(\nu_i \partial_{\nv_i}
    -\av_i\cdot\nv_i \bigr)\,U_i +S_{i,j}U_i =
    P^i\bigl((\nu_{j} \partial_{\nv_i}
    -\av_{j}\cdot\nv_{i})\, U_{j}
    +\widetilde{S}_{i,j}U_{j}\bigr)
    \mbox{ on }\Gamma_{i,j}, j\in {\cal N}_i.\label{eq:dgdiscreteall-2}
  \end{align}
\end{subequations}
We present the analysis first with Robin transmission conditions
(e.g. $q_{i,j}=0$) and general decomposition, and then with order 2
transmission conditions and decomposition in strips.

\subsubsection{The Robin case}
We consider here a general decomposition of the domain, possibly
with corners.  We solve \eqref{eq:algosdorder2} with $q_{i,j}=0$, i.e.
$S_{i,j}U= \widetilde{S}_{i,j}U = p_{i,j}\,U $.

\begin{theorem}
  Assume $q_{i,j}=0$, $p_{j,i}-p_{i,j}-\av_i \cdot \nv_i=0,
  \ p_{i,j}-\frac{\av_i \cdot \nv_i}{2} > 0$. Problem \eqref{eq:dgdiscreteall}
  has a unique solution $(U_i)_{i \in I}$ , and $U_i$ is the limit of
  the iterates of algorithm \eqref{eq:algosdorder2}.
\end{theorem}
\begin{proof}
We first write energy estimates on \eqref{eq:algosdorder2} for $f\equiv 0$ and $u_0\equiv 0$. We start like in the proof of Theorem \ref{th:convcont}. We multiply
\eqref{eq:algosdorder2-1} by $U_i^k$, integrate on $\Omega_i$, then
integrate on the interval $(t_n^i, t_{n+1}^i)$ and use \eqref{eq:interpradau2}
and \eqref{tototiti}:
\begin{multline*}
  \|U_i^k(t^i_{n+1})\|^2_{L^2(\Omega_i)}-\|U_i^k(t^i_{n})\|^2_{L^2(\Omega_i)}\\
  +2\int_{I_n^i}\bigl((\nu_i\nabla U_i^k,\nabla U_i^k)_{L^2(\Omega_i)}
                    + ((c_i+{1 \over 2}\div \av_i) U_i^k,U_i^k)_{L^2(\Omega_i)}\bigr) \,d\tau\\
  +
   \sum_{j\in{\cal N}_i}\int_{I_n^i}\int_{\Gamma_{i,j}}
  \frac{1}{p_{i,j}+p_{j,i}}
  \bigl(\nu_i \partial_{\nv_i} u_i^k
  -{\av_i\cdot\nv_i} U_i^k-p_{j,i}U_i^k\bigr)^2 d\sigma\,d\tau \\
   \leq
  \sum_{j\in{\cal N}_i}\int_{I_n^i}\int_{\Gamma_{i,j}}
  \frac{1}{p_{i,j}+p_{j,i}}
  \bigl(\nu_i \partial_{\nv_i} u_i^k-{\av_i\cdot\nv_i}
  U_i^k +p_{i,j}U_i^k\bigr)^2d\sigma\,d\tau\\
  +\sum_{j\in{\cal N}_i}\int_{I_n^i}\int_{\Gamma_{i,j}}(p_{j,i}-p_{i,j}-\av_i \cdot \nv_i)
  (U_i^k)^2d\sigma \,d\tau.
\end{multline*}
We can not use Gr\"{o}nwall's Lemma like in the continuous case, due to the
presence of the global in time projection operator ${\cal P}^j$ in the transmission
conditions. Therefore we have to assume that $p_{j,i}-p_{i,j}-\av_i \cdot \nv_i=0$
everywhere, which cancels the last term.
%
%
We sum up over the time intervals, using the fact that the errors vanish at time 0:
\begin{multline*}
  \|U_i^k(T)\|^2_{L^2(\Omega_i)}
  +2\min(\nu_0,\alpha_0)\int_{0}^T\|U_i^k\|^2_{H^1(\Omega_i)}\,d\tau\\
  +\sum_{j\in{\cal N}_i}\int_{0}^T
   \int_{\Gamma_{i,j}}
  \frac{1}{p_{i,j}+p_{j,i}}
  \bigl(\nu_i \partial_{\nv_i} u_i^k
  -{\av_i\cdot\nv_i} U_i^k-p_{j,i}U_i^k\bigr)^2 d\sigma\,d\tau \\
  \le
  \sum_{j\in{\cal N}_i}\int_{0}^T\int_{\Gamma_{i,j}}
  \frac{1}{p_{i,j}+p_{j,i}}
  \bigl(\nu_i \partial_{\nv_i} u_i^k
  -{\av_i\cdot\nv_i} U_i^k+p_{i,j}U_i^k\bigr)^2d\sigma\,d\tau.
\end{multline*}
We now insert the transmission conditions
\begin{multline*}
  \|U_i^k(T)\|^2_{L^2(\Omega_i)}+2\min(\nu_0,\alpha_0)
  \int_{0}^T\|U_i^k\|^2_{H^1(\Omega_i)}\,d\tau\\
  +
   \sum_{j\in{\cal N}_i}\int_{0}^T\int_{\Gamma_{i,j}}
   \frac{1}{p_{i,j}+p_{j,i}}
   \bigl(\nu_i \partial_{\nv_i} u_i^k
  -{\av_i\cdot\nv_i} U_i^k-p_{j,i}U_i^k\bigr)^2 d\sigma\,d\tau \\
   \le
  \sum_{j\in{\cal N}_i}\int_{0}^T\int_{\Gamma_{i,j}}
  \frac{1}{p_{i,j}+p_{j,i}}
  \biggl(P^i \bigl(\nu_j
  \partial_{\nv_i}U_j^{k-1}-{\av_j\cdot\nv_i} U_j^{k-1}
  \, +p_{i,j}U_j^{k-1} \bigr)\biggr)^2d\sigma\,d\tau.
\end{multline*}
We use the fact that the projection operator is a contraction to obtain:
\begin{multline*}
  \|U_i^k(T)\|^2_{L^2(\Omega_i)}+2\min(\nu_0,\alpha_0)
   \int_{0}^T\|U_i^k\|^2_{H^1(\Omega_i)}
  \,d\tau\\
  + \sum_{j\in{\cal N}_i}\int_{0}^T\int_{\Gamma_{i,j}}
  \frac{1}{p_{i,j}+p_{j,i}}
  \bigl(\nu_i \partial_{\nv_i} u_i^k
  -{\av_i\cdot\nv_i} U_i^k-p_{j,i}U_i^k\bigr)^2 d\sigma\,d\tau \\
   \le
  \sum_{j\in{\cal N}_i}\int_{0}^T\int_{\Gamma_{j,i}}
  \frac{1}{p_{i,j}+p_{j,i}}
  \bigl(\nu_j \partial_{\nv_j}U_j^{k-1}
  -{\av_j \cdot\nv_j} U_j^{k-1}\, -p_{i,j}U_j^{k-1}\bigr)^2d
  \sigma\,d\tau.
\end{multline*}
We sum up over the subdomains, we define the boundary term
\[
  B^k=
  \sum_{i \in \inter{1,I}}\sum_{j\in{\cal N}_i}\int_{0}^T
  \int_{\Gamma_{i,j}} \frac{1}{p_{i,j}+p_{j,i}}
  \bigl(\nu_i \partial_{\nv_i} u_i^k
  -{\av_i\cdot\nv_i} U_i^k-p_{j,i}U_i^k\bigr)^2 d\sigma\,d\tau,
\]
we obtain
\begin{equation}\label{eq:est5b}
  \sum_{i \in \inter{1,I}}(\|U_i^k(T)\|^2_{L^2(\Omega_i)}
  +2\min(\nu_0,\alpha_0)\int_{0}^T\|U_i^k\|^2_{H^1(\Omega_i)}\,d\tau) +B^k \le   B^{k-1}.
\end{equation}
We first apply this inequality to prove the first part of the Theorem. \eqref{eq:dgdiscreteall} is a square discrete system, and proving well-posedness is equivalent to proving uniqueness. Dropping the superscript in \eqref{eq:est5b} gives the result.
As for the convergence, we proceed as in the continuous case by summing \eqref{eq:est5b} over the iterates to obtain that $\sum_{i \in \inter{1,I}}\|U_i^k(T)\|^2_{L^2(\Omega_i)}$ and
$\sum_{i \in \inter{1,I}}\int_{0}^T\|U_i^k\|^2_{H^1(\Omega_i)}\,d\tau$ tend to zero as $k$
tend to infinity.
\end{proof}

\subsubsection{The Order 2 case}
We restrict ourselves to a splitting of the domain into strips with parallel planar interfaces.
\begin{theorem}\label{th:convdiscRobin}
  We assume that $p_{i,j}=p>0$, $q_{i,j}=q>0$, $s_{i,j}=s>0$,
  $\av_i=0$ and $\rv_{i,j}=0$. Problem \eqref{eq:dgdiscreteall} has a
  unique solution $(U_i)_{i\in I}$, and $U_i$ is the limit of the
  iterates of algorithm \eqref{eq:algosdorder2}.
\end{theorem}
\begin{proof}
We consider the algorithm \eqref{eq:algosdorder2} on the error, so we
suppose $f=u_0=0$. As in the continuous case, the proof is based on energy estimates containing
the term
$$\int_{I_n^i} \int_{\Gamma_{i,j}}
    \left(\nu_i\partial_{\nv_i} u_i^k
  + \mathcal{S}_{i,j} U_i^k\right)^2\,d\sigma\,d\tau,
$$
and that we derive by multiplying successively the first equation of
\eqref{eq:algosdorder2} by the terms $\,U_i^k$, $\partial_t ({\cal
  I}^iU_i^k)$, and $-\Delta_{\Gamma_{i,j}}\, U_i^k$.  We set
$\interleave\varphi\interleave_i^2= \|\sqrt{\nu_i}\,\nabla
\varphi\|^2_{L^2(\Omega_i)} +
\|\sqrt{c_i}\,\varphi\|^2_{L^2(\Omega_i)}$.  We multiply the first
equation of \eqref{eq:algosdorder2} by $\,U_i^k$, we integrate on
$I_n^i \times \Omega_i$ then integrate by parts in space and use
\eqref{eq:interpradau2}:
\begin{multline}\label{eq:estdorder2-1}
 \frac{1}{2}\|U_i^k(t_{n+1}^-)\|^2_i
 +\int_{I_n^i}\interleave U_i^k(\tau,\cdot)\interleave_i^2 \,d\tau
 -\int_{I_n^i} \int_{\Gamma_{i,j}} \nu_i\partial_{\nv_i} u_i^k
 \, U_i^k \, d\sigma\,d\tau
 \le \frac{1}{2}\|U_i^k(t_{n}^-)\|^2_i.\hspace{5mm}
 \end{multline}
We multiply the first equation of \eqref{eq:algosdorder2} by
$\partial_t ({\cal I}^iU_i^k)$, integrate on $I_n^i \times \Omega_i$
and then integrate by parts in space and use \eqref{eq:interpradau2}:
\begin{multline}\label{eq:estdorder2-2}
\frac{1}{2}\interleave U_i^k(t_{n+1}^-)\interleave_i^2
 +\int_{I_n^i} \|\partial_t ({\cal I}^iU_i^k)\|^2_{i} \,d\tau
 - \int_{I_n^i} \int_{\Gamma_{i,j}} \nu_i\partial_{\nv_i} u_i^k
  \, \partial_t ({\cal I}^iU_i^k) \, d\sigma\,d\tau
 \le \frac{1}{2}\interleave U_i^k(t_{n}^-)\interleave_i^2.\hspace{5mm}
\end{multline}
Now we multiply the first equation of \eqref{eq:algosdorder2} by
$-\Delta_{\Gamma_{i,j}}\, U_i^k$ integrate on $I_n^i \times \Omega_i$
and integrate by parts in space and use \eqref{eq:interpradau2}:
\begin{multline}\label{eq:estdorder2-4}
\frac{1}{2}\|  \grads{\Gamma_{i,j}} U_i^k(t_{n+1}^-)\|_i^2
 +\int_{I_n^i}\interleave \grads{\Gamma_{i,j}} U_i^k(t,.) \interleave_i^2\\
 + \int_{I_n^i}\int_{\Gamma_{i,j}} \nu_i\partial_{\nv_i} u_i^k
   \, \Delta_{\Gamma_{i,j}}\, u_i^k  \,d\sigma\,d\tau
 \le \frac{1}{2}\| \grads{\Gamma_{i,j}} U_i^k(t_{n}^-)\|_i^2,
\end{multline}
where we have used the fact that $\Delta_{\Gamma_{i,j}}$ is a
constant coefficient operator.
Let
\begin{equation*}
  E^n(U_i^k)=  \frac{p}{2}\|U_i^k((t^i_{n})^-)\|^2_{i}
   + \frac{q}{2}
  \interleave U_i^k((t^i_{n})^-)\interleave_i^2
  +\frac{sq}{2} \|\grads{\Gamma_{i,j}}U_i^k((t^i_{n})^-)\|^2_{i}.
\end{equation*}
Multiplying \eqref{eq:estdorder2-1} by $p$, \eqref{eq:estdorder2-2} by
$q$ and \eqref{eq:estdorder2-4} by $sq$, and adding the three
equations with \eqref{eq:estdorder2-1}, we get
\begin{multline*}
 E^{n+1}(U_i^k)
 + \int_{I_n^i}[\,
 p\interleave U_i^k(t,\cdot)\interleave_i^2 \,
 +q\|\partial_t({\cal I}_n^iU_i^k)\|^2_{i}\,
 +sq \interleave \grads{\Gamma_{i,j}}U_i^k(t,\cdot)\interleave_i^2 \,]\, dt\\
 -\sum_{j\in {\cal N}_i}\int_{I_n^i}\int_{\Gamma_{ij}}
  \nu_i \partial_{\nv_i} u_i^k\, S_{ij} U_i^k
  dx_2\,dt
  \le E^n(U_i^k).
\end{multline*}
It can be rewritten as
\begin{multline*}
  E^{n+1}(U_i^k)
   +  \int_{I_n^i}[\,
   \interleave U_i^k(t,\cdot)\interleave_i^2 \,
   +q\|\partial_t({\cal I}_n^iU_i^k)\|^2_{i}\,
   +sq \interleave \grads{\Gamma_{i,j}}U_i^k(t,\cdot)\interleave_i^2 \,]\, dt\\
  + \frac{1}{4}\sum_{j\in {\cal N}_i}\int_{I_n^i}\int_{\Gamma_{ij}}
    (\nu_i \partial_{\nv_i} u_i^k-\,S_{ij}U_i^k)^2
     \le
  E^n(U_i^k)
  + \frac{1}{4}\sum_{j\in {\cal N}_i}\int_{I_n^i}\int_{\Gamma_{ij}}
    (\nu_i \partial_{\nv_i} u_i^k+\,S_{ij}U_i^k)^2.
\end{multline*}
We now sum in time for $0 \le n \le N$, and use the transmission
condition.  Since $E^0(U_i^{k})=0$, we obtain
\begin{multline*}
  E^{N+1}(U_i^k)
  + \int_0^T\,[\,
   \interleave U_i^{k}(t,\cdot)\interleave_i^2 \,
   +q\|\partial_t({\cal I}_n^iU_i^{k})\|^2_{i}\,
   +sq \interleave \grads{\Gamma_{i,j}}U_i^{k}(t,\cdot)\interleave_i^2 \,]\, dt\\
  + \frac{1}{4}\sum_{j\in {\cal N}_i}\int_{0}^T\int_{\Gamma_{ij}}
    (\nu_i \partial_{\nv_i} u_i^k-\,S_{ij}U_i^k)^2 \,dt
     \le
   \frac{1}{4}\sum_{j\in {\cal N}_i}\int_{0}^T\int_{\Gamma_{ij}}
    (P^i(-\nu_j \partial_{\nv_j}U_j^{k-1}
   +\,\widetilde{S}_{ij}U_j^{k-1}))^2\,dt.
\end{multline*}
We sum up over the subdomains and use the fact that the projection is a
contraction. Since we are in the case where $p_{ij}=p$, $q_{ij}=q$,
$r_{ij}=0$ and $s_{ij}=s$, we have $\widetilde{S}_{ij}=S_{ji}$. Thus,
we can sum up over the iterates, the boundary terms cancel out, and we
obtain
\begin{multline*}
  \sum_{k=1}^K \sum_{i=1}^I \left(
  E^{N+1}(U_i^k)+ \int_{0}^T[\,
   \interleave U_i^{k}(t,\cdot)\interleave_i^2 \,
   +q\|\partial_t({\cal I}_n^iU_i^{k})\|^2_{i}\,
   +sq \interleave \grads{\Gamma_{i,j}}U_i^{k}(t,\cdot)\interleave_i^2 \,]\, dt
   \right)\\
  + \frac{1}{4}\sum_{i=1}^I\int_{0}^T\int_{\Gamma_{i}}
    (\nu_i \partial_{\nv_i} u_i^k-\,{\cal S}_{ij}U_i^K)^2
    \le\frac{1}{4}\sum_{i=1}^I\int_{0}^T\int_{\Gamma_{i}}
    (\nu_i \partial_{\nv_i} U_i^0 -\,{\cal S}_{ij}U_i^{0})^2.
  \end{multline*}
We conclude as in the proof of Theorem \ref{th:convdiscRobin}.
\end{proof}

We now state the error estimate in the Robin case.

\subsection{Error estimates in the Robin case}

\begin{theorem}\label{th:errorestimate}
  If $\div \av=0$, $p_{i,j}-\frac{\av_i \cdot \nv_i}{2}=p_{j,i}-\frac{\av_j \cdot \nv_j}{2}=p>0$, and $q_{i,j}=0$, the error between
  $\uu$ and the solution $\uud_i$ of \eqref{eq:dgdiscreteall} is estimated by:
 \begin{equation}\label{eq:errestimatesubdomain1}
  \sum_{i=1}^I\|\uu-\uud_i\|^2_{L^\infty(0,T,L^2(\Omega_i))}\le C k^{2(d+1)}
  \|\partial_t^{d+1}\uu\|^2_{L^2(0,T;H^2(\Omega))}.
 \end{equation}
\end{theorem}
\begin{proof} We introduce the projection operator $P_i^-$ as
\begin{equation*}\label{eq:projcont}
\begin{split}
&\forall n \in \inter{1,N_i}, P_i^-\phi \in \PP_{d}(I_n^i),\\
&P_i^-\phi(t_{n+1}^i)=\phi(t_{n+1}^i),\quad \forall \psi_{i,j} \in \PP_{d-1}(I_n^i),
\int_{I_n^i} (P_i^-\phi-\phi)(t)\psi_{i,j}(t)\, dt=0.
\end{split}
\end{equation*}
We define $W_i=P_i^-(u|_{\Omega_i})$, $\Theta_i=\uud_i-W_i$ and
$\rho_i=W_i-u|_{\Omega_i}$. Classical projection estimates
(\cite{thomee:1997:FEM}) yield the estimate on $\rho_i$:
\[
\sum_{i=1}^I\|\rho_i\|^2_{L^\infty(0,T,L^2(\Omega_i))}\le C k^{2(d+1)}  \|\partial_t^{d+1}\uu\|^2_{L^2(0,T;L^2(\Omega))}.
\]
Thus, since $\uud_i-u|_{\Omega_i}=\Theta_i+\rho_i$, it suffices to
prove estimate \eqref{eq:errestimatesubdomain1} for $\Theta_i$. Now,
thanks to the equations on $u$ and $\uud_i$, and the identity
$\frac{d}{dt}\mathcal{I}^iP_i^-=\mathcal{P}^i\frac{d}{dt}$, $\Theta_i$
satisfies:
\begin{eqnarray}\label{toto}
    \begin{array}{c}
    \partial_t ({\cal I}^i\,\Theta_i) +\av\cdot\nabla \Theta_i-\nu\Delta \Theta_i+c\Theta_i
    = -\av\cdot\nabla \rho_i+\nu\Delta \rho_i-c\rho_i\\
    +(1-P^i)(\partial_tu-f) \mbox{ in } \Omega_i\times (0,T), \\
 \ds\bigl(\nu_i\,\partial_{\nv_i}- \av_i\cdot\nv_i \bigr)\Theta_i +
 p_{i,j}\,\Theta_i  =P^i(\bigl(\nu_j\,\partial_{\nv_i}- \av_j\cdot\nv_i \bigr)\Theta_j +
 p_{i,j}\,\Theta_j )\hspace{4.3cm}\\
 -(1-P^i)(\bigl(\nu_j\,\ds\partial_{\nv_i}- \av_j\cdot\nv_i \bigr)W_j +
 p_{i,j}\, W_j ) \mbox{ on } \Gamma_{ij}\times (0,T),\,j\in{\cal N}_i. \\
      \end{array}
\end{eqnarray}
Multiply the first equation of \eqref{toto} by $\,\Theta_i$, integrate
on $(t_n^i, t_{n+1}^i)\times \Omega_i$, using \eqref{eq:interpradau2}
and integrate by parts in space. Terminate with Cauchy Schwarz
inequality:
\begin{multline*}
 \frac{1}{2}\|\Theta_i((t^i_{n+1})^-)\|^2_{i}
 +\int_{I_n^i}\interleave\Theta_i(t,\cdot)\interleave_i^2 \,dt
  -\int_{I_n^i}\int_{\Gamma_{i}}(\nu_i\partial_{\nv_i}\Theta_i
  -\frac{\av_i\cdot\nv_i}{2}\Theta_i)\, \Theta_i  dx_2\,dt \\
   \le
  \frac{1}{2}\|\Theta_i((t^i_{n})^-)\|^2_{i}+C\int_{I_n^i}\|\rho_i(t,\cdot)\|^2_{H^2(\Omega_i)}\,dt.
  \end{multline*}
Rewriting the boundary integral using \eqref{tototiti}, we obtain
\begin{multline*}
 \frac{1}{2}\|\Theta_i((t^i_{n+1})^-)\|^2_{i}
 +\int_{I_n^i}\interleave\Theta_i(t,\cdot)\interleave_i^2 \,dt
  +\frac{1}{4p}\sum_{j\in {\cal N}_i}\int_{I_n^i}\int_{\Gamma_{ij}}
  (\nu_i\partial_{\nv_i}\Theta_i-\av_i\cdot\nv_i\Theta_i-p_{j,i}\Theta_i)^2\, dx_2\,dt \\
   \le
 \frac{1}{4p}\sum_{j\in {\cal N}_i}\int_{I_n^i}\int_{\Gamma_{ij}}
 (\nu_i\partial_{\nv_i}\Theta_i-\av_i\cdot\nv_i\Theta_i+p_{i,j}\Theta_i)^2\, dx_2+
 \frac{1}{2}\|\Theta_i((t^i_{n})^-)\|^2_{i}+C\int_{I_n^i}\|\rho_i(t,\cdot)\|^2_{H^2(\Omega_i)}\,dt.
\end{multline*}
Using the transmission condition in \eqref{toto} together with the
fact that $P^i$ and $1-P^i$ are orthogonal to each other
and have norm 1, we get by a trace theorem
\begin{multline}\label{toto3}
 \frac{1}{2}\|\Theta_i((t^i_{n+1})^-)\|^2_{i}
 +\int_{I_n^i}\interleave\Theta_i(t,\cdot)\interleave_i^2 \,dt
  +\frac{1}{4p}\sum_{j\in {\cal N}_i}\int_{I_n^i}\int_{\Gamma_{ij}}
  (\nu_i\partial_{\nv_i}\Theta_i-\av_i\cdot\nv_i\Theta_i-p_{j,i}\Theta_i)^2\, dx_2\,dt \\
   \hspace{1cm}\le
 \frac{1}{4p}\sum_{j\in {\cal N}_i}\int_{I_n^i}\int_{\Gamma_{ij}}
 (\nu_j\partial_{\nv_j}\Theta_j-\av_j\cdot\nv_j\Theta_j-p_{i,j}\Theta_i)^2\, dx_2+
 \frac{1}{2}\|\Theta_i((t^i_{n})^-)\|^2_{i}\\
 +C\int_{I_n^i}\|\rho_i(t,\cdot)\|^2_{H^2(\Omega_i)}\,dt
 +C\int_{I_n^i}\|(1-P^i)(u|_{\Omega_i})(t,\cdot)\|^2_{H^2(\Omega_i)}\,dt.\hspace{0.7cm}
\end{multline}
Classical error estimates \cite{thomee:1997:FEM}  imply:
\begin{equation*}\label{toto4}
  \int_0^T\|\rho_i(t,\cdot)\|^2_{H^2(\Omega_i)}\,dt
  +\int_0^T\|(1-P^i)(u|_{\Omega_i})(t,\cdot)\|^2_{H^2
  (\Omega_i)}\,\le C k^{2(d+1)}  \|\partial_t^{d+1}\uu\|^2_{L^2(0,T;H^{2}(\Omega_i))}.
\end{equation*}
Summing \eqref{toto3} in $i$ and $n$, and using the previous equation yields
\eqref{eq:errestimatesubdomain1}.
\end{proof}
\section{Space-time nonconforming algorithm}
%
In this section we describe the implementation of algorithm
\eqref{eq:algosdorder2}, especially in the cases $d=0$ and $d=1$.  We
start from the semi-disrete in time scheme and use finite elements for
the space discretization in each subdomain. In order to permit
non-matching grids in time and space on the boundary, we extend the
nonconforming approach in \cite{gander:2005:NZF}.

We describe first the implementation of algorithm
\eqref{eq:algosdorder2} at the semi-discrete in time level, and then at
the space-time discret level.

\subsection{Time discretization}
We recall the subdomain scheme in time, and give it in details for $d=0$ and $d=1$. Then
we describe the computation of the transmission conditions in algorithm
\eqref{eq:algosdorder2}.

\subsubsection{Interior scheme}\label{sec:interiorscheme}
We consider the subdomain problem in the algorithm
\eqref{eq:algosdorder2} at iteration $k$ in $\co=\Omega_i$.  Let
${\cal B}_i=H_1^1(\Omega_i)$ if $q > 0$, ${\cal B}_i=H^1(\Omega_i )$
if $q = 0$.  We set $U=U_i^k \in {\cal P}_d({\cal B}_i,{\cal T}_i)$,
and we omit the subscript $i$ for the local time scheme to simplify
the notations :
\begin{eqnarray}\label{eq:dGstronginterpd}
  \begin{array}{c}
    \partial_t ({\cal I} U ) +\div(\av U -\nu\nabla U)+cU= P f \mbox{ in }
    \co\times (0,T) \\
    \ds\bigl(\nu\,\partial_{\nv}- \av\cdot\nv \bigr)U +
    p\,U + q(\partial_t ({\cal I} U)   + \divs{\Gamma}(\rv U -s\grads{\Gamma}U))
    =Pg \mbox{ on } \Gamma\times (0,T).
  \end{array}
\end{eqnarray}
%

\textbf{Case $d=0$}
\vspace{2mm}

In the case $d=0$, the approximating functions are piecewise constant
in time, then $U(t)=U^{n+1}=U^{n}_+$ in $I_n$, we have ${\cal
  I}_nU=U^{n}+\frac{t-t_{n}}{k_n}(U^{n+1}-U^{n})$,
$P\xi=\frac{1}{k_n}\int_{I_n}\xi(\cdot,s)\,ds$ and the method reduces
to the modifed backward Euler method
\begin{eqnarray}\label{eq:dGstronginterpd0}
  \begin{array}{c}
     \ds\frac{U^{n+1}}{k_n}+\div(\av U^{n+1}-\nu\nabla U^{n+1})+cU^{n+1}
     \ds= \frac{U^{n}}{k_n}+\frac{1}{k_n}\int_{I_n}f(\cdot,s)\,ds\mbox{ in }\co \\
    \ds\bigl(\nu\,\partial_{\nv}- \av\cdot\nv \bigr)U^{n+1} +
    p\,U^{n+1} + q(\frac{U^{n+1}}{k_n}  + \divs{\Gamma}(\rv U^{n+1} -s\grads{\Gamma}U^{n+1}))\\
    \ds= q\frac{U^{n}}{k_n} +\frac{1}{k_n}\int_{I_n}g(\cdot,s)\,ds\mbox{ on } \Gamma.
  \end{array}
\end{eqnarray}

\textbf{Case $d=1$}
\vspace{2mm}

In that case, for piecewise linear functions of $t$, using a basis of
Legendre polynomials we may write,
$U(t)=U^{n+1}_0+2\frac{t-t_{n+{1/2}}}{k_n}U_1^{n+1}$, on $I_n^i$,
with $t_{n+{1/2}}=\frac{t_{n}+t_{n+1}}{2}$, $U^n=U(t_n)$, and we have on $I_n^i$ :
\begin{eqnarray*}
    \ds{\cal I}_nU={1 \over 4}(5 U_0^{n+1}-U_1^{n+1}-U^{n})
     +(U_0^{n+1}+U_1^{n+1}-U^n) \bigl(\frac{t-t_{n+{1/2}}}{k_n}\bigr)\nonumber\\
     +3(-U_0^{n+1}+U_1^{n+1}+U^{n})
     \bigl(\frac{t-t_{n+{1/2}}}{k_n}\bigr)^2,
\end{eqnarray*}
and $P\xi=\xi_0+2\frac{t-t_{n+1/2}}{k_n}\xi_1$ with
\begin{eqnarray}\label{eq:systprojd1}
   \left\{\begin{array}{l}
      \xi_0=\frac{1}{k_n}\int_{I_n}\xi(\cdot,s)\,ds\\
      \xi_1=\frac{6}{k_n}\int_{I_n}\frac{s-t_{n+1/2}}{k_n}\xi(\cdot,s)\,ds
    \end{array}\right.
\end{eqnarray}
Thus, we obtain for the determination
of $U^{n+1}_0$ and $U^{n+1}_1$ the system
\begin{equation}\label{eq:dGstronginterpd1}
  \begin{array}{l}
    \begin{array}{l}
     \ds\frac{1}{k_n}(U_0^{n+1}+U_1^{n+1})+\div(\av U_0^{n+1}
     -\nu\nabla U_0^{n+1})+cU_0^{n+1} = \frac{U^{n}}{k_n}+\frac{1}{k_n}\int_{I_n}f(\cdot,s)\,ds
    \\
      \ds\frac{3}{k_n}(-U_0^{n+1}+U_1^{n+1})+\div(\av U_1^{n+1}
         -\nu\nabla U_1^{n+1})+cU_1^{n+1}\\
      \hspace{30mm}\ds= -\frac{3U^{n}}{k_n}
      +\frac{6}{k_n}\int_{I_n}\frac{s-t_{n+1/2}}{k_n}f(\cdot,s)\,ds\mbox{ in }    \co
    \end{array}
    \\
    \begin{array}{l}
      \ds\bigl(\nu\,\partial_{\nv}- \av\cdot\nv \bigr)U_0^{n+1} +
    p\,U_0^{n+1} + q(\frac{1}{k_n}(U_0^{n+1}+U_1^{n+1}) + \divs{\Gamma}(\rv U_0^{n+1} -s\grads{\Gamma}U_0^{n+1}))
    \\
     \hspace{30mm} \ds= q\frac{U^{n}}{k_n} +\frac{1}{k_n}\int_{I_n}g(\cdot,s)\,ds
      \\
      \ds\bigl(\nu\,\partial_{\nv}- \av\cdot\nv \bigr)U_1^{n+1} +
    p\,U_1^{n+1} + q(\frac{3}{k_n}(-U_0^{n+1}+U_1^{n+1}) + \divs{\Gamma}(\rv U_1^{n+1} -s\grads{\Gamma}U_1^{n+1}))
    \\
     \hspace{30mm} \ds= -q\frac{3U^{n}}{k_n} +\frac{6}{k_n}\int_{I_n}\frac{s-t_{n+1/2}}{k_n}g(\cdot,s)\,ds\mbox{ on } \Gamma.
  \end{array}
  \end{array}
\end{equation}

Multiplying the first equation of \eqref{eq:dGstronginterpd0} by $v \in {\cal B}_i$
(resp. the first equation of \eqref{eq:dGstronginterpd1} by $v \in {\cal B}_i$ and
the second equation of \eqref{eq:dGstronginterpd1} by $w \in {\cal B}_i$), integrating by parts on $\co$,
and using the boundary conditions, the variational formulation is:
\vspace{3mm}

\textbf{Case $d=0$ (Variational formulation)}
%
\begin{multline}\label{eq:semidiscreteq0}
  m(U^{n+1},v)+k_n a(U^{n+1},v)=\\m(U^{n},v)
  +\int_{I_n^i}(f(\cdot,s),v)ds+\int_{I_n^i}(g(\cdot,s),v)_\Gamma ds,
  \quad \forall v \in {\cal B}_i.
\end{multline}

\textbf{Case $d=1$ (Variational formulation)}
%
\begin{multline}\label{eq:semidiscreteq1}
m(U_0^{n+1},v)+k_n a(U_0^{n+1},v)+m(U_1^{n+1},v)\\
  =m(U^{n},v)+\int_{I_n}(f(\cdot,s),v)ds+\int_{I_n}(g(\cdot,s),v)_\Gamma ds,\\
\hspace{-5cm} -m(U_0^{n+1},v)+m(U_1^{n+1},w)+k_n a(U_1^{n+1},w)\\
  \hspace{1cm}=-m(U^{n},v)+\int_{I_n}\frac{2(s-t_{n+1/2})}{k_n}(f(\cdot,s),w)ds\\
    +\int_{I_n}\frac{2(s-t_{n+1/2})}{k_n}(g(\cdot,s),w)_\Gamma ds, \quad \forall v,\,w \in {\cal B}_i.
\end{multline}
\begin{remark}
Equations \eqref{eq:semidiscreteq0} and \eqref{eq:semidiscreteq1}
can be derived directly from \eqref{eq:dGvariat}.  However we will need formulas \eqref{eq:dGstronginterpd0} and \eqref{eq:dGstronginterpd1} in the space nonconforming case. \\
\end{remark}

We now discuss the computation of the right-hand side on the interface
$\Gamma_{i,j}\times (0,T)$ for $j\in {\cal N}_i$ in the algorithm
\eqref{eq:algosdorder2}.

\subsubsection{Transmission terms}\label{sec:boundaryterms}
Let $(g^1_{i,j})$ be a given initial guess
in ${\cal P}_d(L^2(\Gamma_{i,j}),{\cal T}_i)$, for $1 \le i \le I$, $j \in {\cal N}_i$.
Then, at iteration $k \ge 1$, we solve the subdomain problem in $\Omega_i$ :
\begin{subequations}
  \begin{align}
    \label{eq:algosdorder2d-1}
    & \partial_t ({\cal I}^iU_i^k)
     +\div(\av_i U_i^k -\nu_i\nabla U_i^k)
    +c_i\,U_i^k=P^if\mbox{ in }\Omega_i\times(0,T),\\
    \label{eq:algosdorder2d-2}
    &\ds\bigl(\nu_i \partial_{\nv_i}
    -\av_i\cdot\nv_i \bigr)\,U_i^{k} +S_{i,j}U_i^{k} = g^{k}_{i,j}, \ j\in {\cal N}_i
    \mbox{ on }\Gamma_{i,j}\times(0,T),
  \end{align}
\end{subequations}
The function $g^k_{i,j}$ is defined for $k \ge 2$ by
\begin{equation}\label{eq:gk}
   g^k_{i,j}=P^i \tilde{g}^k_{j,i},
\end{equation}
with $\tilde{g}^k_{j,i}$, $k \ge 2$, defined by
\begin{equation*}\label{eq:gktilde}
\tilde{g}^k_{j,i}= \bigl(-(\nu_{j} \partial_{\nv_j}
    -\av_{j}\cdot\nv_{j})\, U_{j}^{k-1}
    +\widetilde{S}_{i,j}U_{j}^{k-1}\bigr).
\end{equation*}
We remark that, for $k \ge 2$,
\begin{multline}
  \tilde{g}^k_{j,i}=-g^{k-1}_{j,i}+S_{j,i}U_{j}^{k-1}+\widetilde{S}_{i,j}U_{j}^{k-1}\\
               \hspace{-2cm}=-g^{k-1}_{j,i}+(p_{i,j}+p_{j,i})U_{j}^{k-1}
    +(q_{i,j}+q_{j,i})\partial_t ({\cal I}^jU_{j}^{k-1})\nonumber\\
  \hspace{2mm}+ q_{i,j}\,(\divs{\Gamma_{i,j}}(\rv_{i,j}U_{j}^{k-1}-s_{i,j}\grads{\Gamma_{i,j}}U_{j}^{k-1}))
   + q_{j,i}\,(\divs{\Gamma_{j,i}}(\rv_{j,i}U_{j}^{k-1}-s_{j,i}\grads{\Gamma_{j,i}}U_{j}^{k-1})).\nonumber
\end{multline}
Once $\tilde{g}^k_{j,i}$ is computed from $U_{j}^{k-1}$, we obtain
$g^k_{i,j}$ from \eqref{eq:gk} as follows : we introduce the basis
functions $(\phi_{n,\alpha}^i)_{0 \le \alpha \le d}$ of polynomial of
degree lower than $d$ on subinterval $I_n^i$, then
\begin{equation*}\label{eq:gkIn}
(g^k_{i,j})_{|I_n^i}=(P^i \tilde{g}^k_{j,i})_{|I_n^i}
  =\sum_{\alpha=0}^d G^{i,k}_{n,\alpha} \phi_{n,\alpha}^i
\end{equation*}
with $G^{i,k}_{n,\alpha} \in L^2(\Gamma_{i,j})$ solution of the system
\begin{equation*}\label{eq:systGproj}
\sum_{\alpha=0}^d G^{i,k}_{n,\alpha} \int_{I_n^i}\phi_{n,\alpha}^i\phi_{n,\beta}^i \,ds
=\int_{I_n^i} \tilde{g}^k_{j,i} \phi_{n,\beta}^i, \quad \beta \in \{0,...,d\}.
\end{equation*}
Thus, the computation of $g^k_{i,j}$ on each $I_n^i$ needs the
computation of terms in the form
\begin{equation}\label{eq:inttildeg}
\int_{I_n^i} \tilde{g}^k_{j,i} \phi_{n,\beta}^i\,ds,
\end{equation}
for $\beta \in \{0,...,d\}$. Recall that $\tilde{g}^k_{j,i}$ is defined on $\Gamma_{i,j}\times I$ and
$\tilde{g}^k_{j,i} \in {\cal P}_d(L^2(\Gamma_{i,j}),{\cal T}_j)$. Thus,
we first write the integral in \eqref{eq:inttildeg} as an integral over $I$ :
let $\Phi_{n,\alpha}^i$ be the function defined on $I$, equal to
$\phi_{n,\alpha}^i$ on $I_n^i$ and equal to zero on $I\backslash I_n^i$. Then
\begin{equation}\label{eq:intItildeg}
 \ds\int_{I_n^i} \tilde{g}^k_{j,i} \phi_{n,\beta}^i\,ds
  =\int_{I} \tilde{g}^k_{j,i} \Phi_{n,\beta}^i\,ds.
\end{equation}
We now decompose $\tilde{g}^k_{j,i}$ on the basis
functions $(\phi_{m,\alpha}^j)_{0 \le \alpha \le d}$ of polynomial of
degree lower than $d$ on each subinterval $I_m^j$ :
\begin{equation*}
(\tilde{g}^k_{j,i})_{|I_m^j}= \sum_{\alpha=0}^d \tilde{G}^{j,k}_{m,\alpha}\phi_{m,\alpha}^j,
\end{equation*}
with $\tilde{G}^{j,k}_{m,\alpha}\in L^2(\Gamma_{i,j})$ solution of the system
\begin{equation*}\label{eq:systGlocal}
\sum_{\alpha=0}^d \tilde{G}^{j,k}_{m,\alpha}\int_{I_m^j}\phi_{m,\alpha}^j\phi_{m,\beta}^j \,ds
=\int_{I_m^j} \tilde{g}^k_{j,i} \phi_{m,\beta}^j, \quad \beta \in \{0,...,d\}.
\end{equation*}
Introducing the function $\Phi_{m,\alpha}^j$ defined on $I$, equal to
$\phi_{m,\alpha}^j$ on $I_m^j$ and equal to zero on $I\backslash I_m^j$, we have
\begin{equation}\label{eq:tildeg}
\tilde{g}^k_{j,i}= \sum_{m=0}^{N_j}\sum_{\alpha=0}^d \tilde{G}^{j,k}_{m,\alpha}\Phi_{m,\alpha}^j,
\end{equation}
Replacing \eqref{eq:tildeg} in \eqref{eq:intItildeg} leads to
\begin{equation*}
  \int_{I_n^i} \tilde{g}^k_{j,i} \phi_{n,\beta}^i\,ds
   = \ds\sum_{m=0}^{N_j}\sum_{\alpha=0}^d \tilde{G}^{j,k}_{m,\alpha}
     \int_{I} \Phi_{m,\alpha}^j \Phi_{n,\beta}^i\,ds.
\end{equation*}
Let ${\mathbb M}^{\alpha,\beta}$ be the projection matrix defined by
\begin{equation*}
({\mathbb M}^{\alpha,\beta})_{n+1,m+1}=\int_{I} \Phi_{m,\alpha}^j \Phi_{n,\beta}^i\,ds.
  \quad 0\le ,n\le N_i,\, 0 \le m \le N_j.
\end{equation*}
%
%
Then we have, for $0 \le n \le N_i$,
\begin{equation*}
\int_{I_n^i} \tilde{g}^k_{j,i} \phi_{n,\beta}^i\,ds
  = \ds \sum_{\alpha=0}^d  ({\mathbb M}^{\alpha,\beta} \tilde{\mathbf G}^{j,k}_{\alpha})_n
\end{equation*}
with $\tilde{\mathbf G}^{j,k}_{\alpha}=(\tilde{G}^{j,k}_{0,\alpha},...,\tilde{G}^{j,k}_{N_j,\alpha})^t$.\\

In the special cases $d=0$ and $d=1$, we obtain :
\vspace{3mm}

\textbf{Case $d=0$}
\vspace{2mm}

In that case there is one basis function $\phi^i_{n,0}=1$ on $I_n^i$, and
\begin{equation*}
\int_{I_n^i} \tilde{g}^k_{j,i} \phi_{n,0}^i\,ds
  = \int_{I_n^i} \tilde{g}^k_{j,i} \,ds\ds
  =({\mathbb M}^{0,0} \tilde{\mathbf G}^{j,k}_{0})_n,
\end{equation*}
with $({\mathbb M}^{0,0})_{n+1,m+1}=\int_I \mathbf{1}_{I_m^j}\mathbf{1}_{I_n^i}\,ds$,
$\tilde{\mathbf G}^{j,k}_{m,0}=\frac{1}{k_m^j}\int_{I_m^j} \tilde{g}^k_{j,i} \,ds$,
$0\le ,n\le N_i,\, 0 \le m \le N_j$.
\vspace{3mm}

\textbf{Case $d=1$}
\vspace{2mm}

In that case there are two basis functions $\phi^i_{n,0}=1,
\ \phi^i_{n,1}=2\frac{s-t_{n+1/2}^i}{k_n^i}$ on $I_n^i$, and
\begin{eqnarray*}
  \begin{array}{l}
\int_{I_n^i} \tilde{g}^k_{j,i} \phi_{n,0}^i\,ds
  =({\mathbb M}^{0,0} \tilde{\mathbf G}^{j,k}_{0}+{\mathbb M}^{1,0} \tilde{\mathbf G}^{j,k}_{1})_n,\\
\int_{I_n^i} \tilde{g}^k_{j,i} \phi_{n,1}^i\,ds
  =({\mathbb M}^{0,1} \tilde{\mathbf G}^{j,k}_{0}+{\mathbb M}^{1,1} \tilde{\mathbf G}^{j,k}_{1})_n,
  \end{array}
\end{eqnarray*}
with, for $0\le ,n\le N_i,\, 0 \le m \le N_j$,
\begin{eqnarray*}
  \begin{array}{l}
\ds({\mathbb M}^{0,0})_{n+1,m+1}=\int_I \mathbf{1}_{I_m^j}\mathbf{1}_{I_n^i}\,ds, \quad
({\mathbb M}^{1,1})_{n+1,m+1}=4\int_I \frac{s-t_{m+1/2}^j}{k_m^j}\mathbf{1}_{I_m^j}
                           \frac{s-t_{n+1/2}^i}{k_n^i}\mathbf{1}_{I_n^i}\,ds, \\
\ds({\mathbb M}^{1,0})_{n+1,m+1}=2\int_I \frac{s-t_{m+1/2}^j}{k_m^j}\mathbf{1}_{I_m^j}\mathbf{1}_{I_n^i}\,ds,\quad
\ds({\mathbb M}^{0,1})_{n+1,m+1}=2\int_I \mathbf{1}_{I_m^j}\frac{s-t_{n+1/2}^i}{k_n^i}\mathbf{1}_{I_n^i}\,ds,
  \end{array}\nonumber
\end{eqnarray*}
and  $\tilde{\mathbf G}^{j,k}_{m,0}, \ \tilde{\mathbf G}^{j,k}_{m,1}$ defined by
\begin{eqnarray*}
  \left\{\begin{array}{l}
      \tilde{\mathbf G}^{j,k}_{m,0}=\frac{1}{k_m^j}\int_{I_m^j}\tilde{g}^k_{j,i} \,ds\\\
      \tilde{\mathbf G}^{j,k}_{m,1}=\frac{6}{k_m^j}\int_{I_m^j}\frac{s-t_{m+1/2}^j}{k_m^j}\tilde{g}^k_{j,i} \,ds
    \end{array}\right.
\end{eqnarray*}
The projection matrices ${\mathbb M}^{\alpha,\beta}$ are computed by a
simple and optimal projection algorithm without any additional grid
(see \cite{gander:2003:OSWW},\cite{gander:2005:NZF}).\\

We now discuss the space dicretization using finite elements.
\subsection{Space discretization}
We suppose that each subdomain $\Omega_i$ is provided with
its own mesh ${\cal T}_h^i, \ 1 \le i \le I$, such that
\begin{equation*}
\overline \Omega_i=\cup_{T \in {\cal T}_h^i} T.
\end{equation*}
For $T \in {\cal T}_h^i$, let $h_T:=\sup_{x,y \in T} d(x,y)$ be the diameter of $T$
and $h$ the discretization parameter
\begin{equation*}
h=\max_{1 \le i\le I} h_i, \quad \mbox{with} \quad
h_i=\max_{T \in {\cal T}_h^i} h_T.
\end{equation*}
%
%
Let ${\cal P}_1(T)$ denote the space of all polynomials defined over T
of total degree less than or equal to $1$.
Then, we define over each subdomain the conforming spaces $V_h^i$  by :
\begin{equation*}
V_h^i=\{v_{i,h} \in {\cal C}^0(\overline \Omega_i),
\ \  {v_{i,h}}_{|T} \in {\cal P}_1(T), \ \forall T \in {\cal T}_h^i\}.
\end{equation*}
In what follows we assume that the mesh is designed by taking into account
the geometry of the $\Gamma_{i,j}$ in the sense that, the
space of traces over each
$\Gamma_{i,j}$ of elements of $V_h^i$ is a finite element space
denoted by ${\cal V}_h^{i,j}$. Let $n^{i,j}$ be the dimension of ${\cal V}_h^{i,j}$ and
$(\chi_{\ell,h}^{i,j})_{1\le \ell \le n^{i,j}}$ the finite element basis functions of ${\cal V}_h^{i,j}$.
%

We consider two cases : when the grids in space are conforming, and
the case of nonconforming space grids.

\subsubsection{Conforming case}
In the case of conforming grids in space, we have ${\cal V}_h^{i,j}={\cal V}_h^{j,i}$.
We can replace ${\cal B}_i$ by $V_h^i$ in the variational formulation.
We set :
\begin{equation}\label{eq:atilde}
  \tilde{a}_i(u,v)=
  \int_{\Omega_i}(\frac{1}{2}((\av_i\cdot\nabla u)v-(\av_i\cdot\nabla v)u))\, dx
  +\int_{\Omega_i} \nu_i \nabla u\cdot\nabla v \, dx + \int_{\Omega_i} (c_i+\frac{1}{2}\div \av_i) u v \, dx,
\end{equation}
and
\begin{equation*}
\begin{array}{l}
  <C_{i,j}u,v>_{\Gamma_{i,j}}
  =\int_{\Gamma_{i,j}} \bigl((p_{i,j}-\frac{\av_i\cdot \nv_i}{2})\,uv \\
    \hspace{30mm}+ q_{i,j}\, (\partial_t ({\cal I}^i u)
    + \divs{\Gamma_{i,j}}(\rv_{i,j}u))v-s_{i,j}\grads{\Gamma_{i,j}}u\grads{\Gamma_{i,j}}v\bigr)\,d\sigma,\\
  <\widetilde{C}_{i,j}u,v>_{\Gamma_{i,j}}
  =\int_{\Gamma_{i,j}} \bigl((p_{i,j}-\frac{\av_i\cdot \nv_i}{2})\,uv \\
    \hspace{30mm}+ q_{i,j}\, (\partial_t ({\cal I}^j u)
    + \divs{\Gamma_{i,j}}(\rv_{i,j}u))v-s_{i,j}\grads{\Gamma_{i,j}}u\grads{\Gamma_{i,j}}v\bigr)\,d\sigma.
\end{array}
\end{equation*}
We introduce the discret algorithm : let $(g^1_{i,j,h})$ be a given initial guess
in ${\cal P}_d({\cal V}_h^{i,j},{\cal T}_i)$, for $1 \le i \le I$, $j \in {\cal N}_i$.
 Let $U_{i,h}^k$ be the approximation of $u_i^k$ in ${\cal P}_d(V_h^{i},{\cal T}_j)$.
Then, at iteration $k \ge 1$, we solve the subdomain problem in $\Omega_i$:
\begin{multline}\label{eq:algosdorder2discreteconf}
  \int_{\Omega_i}\bigl(\partial_t ({\cal I}^iU_{i,h}^k)v_{i,h}+\tilde{a}_i(U_{i,h}^k,v_{i,h})\bigr)\,dx\,
     +<C_{i,j} U_{i,h}^k,v_{i,h}>_{\Gamma_{i,j}}\\
    =\int_{\Omega_i}P^ifv_{i,h}\,dx+\int_{\Gamma_{i,j}}g_{i,j,h}^kv_{i,h}\,d\sigma, \mbox{ in } (0,T), \
    \forall v_{i,h} \in V_h^i,
\end{multline}
For $k \ge 2$, $v_h \in {\cal V}_h^{i,j}$, we define
\begin{equation}\label{eq:gijkh}
  \int_{\Gamma_{i,j}}g_{i,j,h}^kv_h\,d\sigma:=P^i\int_{\Gamma_{i,j}}\tilde{g}_{j,i,h}^kv_h\,d\sigma,
\end{equation}
with
\begin{equation*}
\int_{\Gamma_{i,j}}\tilde{g}_{j,i,h}^kv_h\,d\sigma:=-\int_{\Gamma_{i,j}}g^{k-1}_{j,i,h}v_h\,d\sigma
  +<C_{j,i}U_{j,h}^{k-1}+\widetilde{C}_{i,j}U_{j,h}^{k-1},v_h>_{\Gamma_{i,j}}.
\end{equation*}
In equation \eqref{eq:gijkh}  we used the fact that the space of traces over each
$\Gamma_{i,j}$ of elements of $V_h^i$ is the same as the space of traces over each
$\Gamma_{i,j}$ of elements of $V_h^j$. For the computation of the right-hand side in
\eqref{eq:algosdorder2discreteconf}, we follow the same steps as in section \ref{sec:boundaryterms},
where we replace $\tilde{g}_{j,i}^k$ with $\tilde{\gv}_{j,i,h}^k \in {\cal V}_h^{i,j}$ defined by
\begin{equation*}
\tilde{\gv}_{j,i,h}^k=\bigl(\int_{\Gamma_{i,j}}\tilde{g}_{j,i,h}\chi^{i,j}_{1,h}\,d\sigma,...,
       \int_{\Gamma_{i,j}}\tilde{g}_{j,i,h}\chi^{i,j}_{n^{i,j},h}\,d\sigma\bigr)^t,
\end{equation*}
and we replace $\tilde{G}^{i,k}_{m,\alpha}$ with $\tilde{G}^{i,k}_{m,\alpha,h} \in {\cal V}_h^{i,j}$ solution of
\begin{equation}\label{eq:systGlocald}
\sum_{\alpha=0}^d \tilde{G}^{j,k}_{m,\alpha,h}\int_{I_m^j}\phi_{m,\alpha}^j\phi_{m,\beta}^j \,ds
=\int_{I_m^j} \tilde{\gv}^k_{j,i,h} \phi_{m,\beta}^j, \quad \beta \in \{0,...,d\}.
\end{equation}
The discrete formulation in the cases $d=0$ and $d=1$ are obtained from \eqref{eq:semidiscreteq0} and
\eqref{eq:semidiscreteq1}, by replacing ${\cal B}_i$ by $V_h^i$.\\

When the space grids are nonconforming, following
\cite{gander:2005:NZF}, we cannot replace directly ${\cal B}_i$ by the
finite element space $V_h^i$ in the variational formulation.  We have
to consider equation \eqref{eq:dGstronginterpd}
(i.e. \eqref{eq:dGstronginterpd0} for $d=0$, and
\eqref{eq:dGstronginterpd1} for $d=1$).
\subsubsection{Nonconforming case}
In this section we extend the nonconforming approach in \cite{gander:2005:NZF}.
We consider the mortar spaces $\tilde{W}_h^{i,j}$ as in \cite{gander:2005:NZF}.
Let $m^{i,j}$ be the dimension of $\tilde{W}_h^{i,j}$ and $(\psi^{i,j}_{k,h})_{1\le k\le m^{i,j}}$ the finite element
basis functions of $\tilde{W}_h^{i,j}$.
We introduce the discrete algorithm : let $(U_{i,h}^{k-1}, Q_{i,h}^{k-1}) \in
{\cal P}_d(V_h^i,{\cal T}_i) \times {\cal P}_d(\tilde{W}_h^{i,j},{\cal T}_i)$
be a discrete approximation of $(U_{i}^{k-1},\nu_i\partial_{\nv_i}U^{k-1}_i)$
in $\Omega_i$ at step $k-1$.
Then  $(U_{i,h}^{k}, Q_{i,h}^{k})$ is the solution in
${\cal P}_d(V_h^i,{\cal T}_i) \times {\cal P}_d(\tilde{W}_h^{i,j},{\cal T}_i)$ of
\begin{eqnarray}\label{eq:algosdorder2discrete}
  \begin{array}{l}
 \ds \frac{d\,}{dt}\,({\cal I}^iU_{i,h}^k, v_{i,h})_i+\tilde{a}_i(U_{i,h}^k,v_{i,h})_i\\
 \hspace{2cm}+\ds\int_{\Gamma_{i,j}} (Q_{i,h}^k-\frac{\av_i\cdot\nv_i}{2} U_{i,h}^k)v_{i,h} d\sigma
   = (P^if,v_{i,h})_i, \mbox{ in } (0,T), \ \forall v_{i,h} \in V^i_h,\\[5mm]
\ds\int_{\Gamma_{i,j}}\bigl(Q_{i,h}^k- \av_i\cdot\nv_i U_{i,h}^k
     + p_{i,j}\,U_{i,h}^k\bigr)\psi_h^{i,j}\,d\sigma\\
    \ds+ \int_{\Gamma_{i,j}}\bigl(q_{i,j}(\partial_t ({\cal I}^i U_{i,h}^k)
    + \divs{\Gamma_{i,j}}(\rv_{i,j} U_{i,h}^k))\psi_h^{i,j}
    + q_{i,j}s_{i,j}\grads{\Gamma_{i,j}}U_{i,h}^k\grads{\Gamma_{i,j}}\psi_h^{i,j}\bigr)\,d\sigma \\
    \hspace{2cm}\ds=\int_{\Gamma_{i,j}}P^i\bigl(-Q_{j,h}^{k-1}
    -\av_{j}\cdot\nv_{i} U_{j,h}^{k-1} + p_{i,j}\,U_{j,h}^{k-1}\bigr)\psi_h^{i,j}\,d\sigma\\
   \ds + \int_{\Gamma_{i,j}}P^i\bigl(q_{i,j}(\partial_t ({\cal I}^j U_{j,h}^{k-1})
    + \divs{\Gamma_{i,j}}(\rv_{i,j} U_{j,h}^{k-1}))\psi_h^{i,j}
    + q_{i,j}s_{i,j}\grads{\Gamma_{i,j}}U_{j,h}^{k-1}\grads{\Gamma_{i,j}}\psi_h^{i,j}\bigr)\,d\sigma\\
    \mbox{ on } (0,T), \quad \forall \psi_h^{i,j} \in \tilde{W}_h^{i,j},j\in {\cal N}_i.
  \end{array}
\end{eqnarray}
We give first the interior scheme for $d=0$ and $d=1$ and then the computation of
the right-hand side in the transmission condition of \eqref{eq:algosdorder2discrete}.\\

{\bf Interior scheme.}
The discrete problem in subdomain $\co=\Omega_i$ in \eqref{eq:algosdorder2discrete} is
defined as follows : find
$(U_h,Q_h):=(U_{i,h}^{k}, Q_{i,h}^{k})$ in
${\cal P}_d(V_h^i,{\cal T}_i) \times {\cal P}_d(\tilde{W}_h^{i,j},{\cal T}_i)$ solution of
\begin{eqnarray}\label{eq:formvariatdiscrete}
  \begin{array}{l}
 \ds \frac{d\,}{dt}\,({\cal I} U_h, v_h)+\tilde{a}(U_h,v_h)
 +\int_\Gamma (Q_h- \frac{\av\cdot\nv}{2} U_h)v_h d\sigma
   = (Pf,v_h), \ \mbox{ in } (0,T), \ \forall v_h \in V^i_h, \\
\ds\int_\Gamma\bigl((Q- \av\cdot\nv U_h + p\,U_h
    + q(\partial_t ({\cal I} U_h)   + \divs{\Gamma}(\rv U_h)))\psi_h
    +qs\grads{\Gamma}U_h\grads{\Gamma}\psi_h\bigr)\,d\sigma\\
    =\int_\Gamma (Pg)\psi_h\,d\sigma, \  \mbox{ on } (0,T), \ \forall \psi_h \in \tilde{W}_h^{i,j}.
  \end{array}\nonumber
\end{eqnarray}
In the cases $d=0$ and $d=1$ we obtain
\vspace{3mm}

\textbf{Case $d=0$}
\vspace{2mm}

In that case, the approximating functions are piecewise constant
in time: $U_h(t)=U_h^{n+1}=U_{h,+}^{n}$ and $Q_h(t)=Q_h^{n+1}=Q_{h,+}^{n}$ on $I_n^i$, and
the discrete problem reduces to find $(U_h^{n+1},Q_h^{n+1}) \in V^i_h \times \tilde{W}_h^{i,j}$
solution of
\begin{eqnarray}\label{eq:dGstronginterpd0-2}
  \begin{array}{c}
     \ds(\frac{U_h^{n+1}}{k_n},v_h)+\tilde{a}(U_h^{n+1},v_h)
 +\int_\Gamma (Q_h^{n+1}- \frac{\av\cdot\nv}{2} U_h^{n+1})v_h d\sigma\\ \ds= (\frac{U_h^{n}}{k_n},v_h)
 +\frac{1}{k_n}\int_{I_n}(f(\cdot,s),v_h)\,ds, \forall v_h \in V^i_h, \\
    \ds\int_\Gamma\bigl( Q_h^{n+1}- \av\cdot\nv U_h^{n+1} +
    p\,U_h^{n+1} + q(\frac{U_h^{n+1}}{k_n}  + \divs{\Gamma}(\rv U_h^{n+1}))\psi_h
    +qs\grads{\Gamma}U_h^{n+1}\grads{\Gamma}\psi_h\bigr) \,d\sigma\\
    \ds= \int_\Gamma q\frac{U_h^{n}}{k_n}\psi_h \,d\sigma
    +\frac{1}{k_n}\int_{I_n}\int_\Gamma g(s)\psi_h \, d\sigma\,ds,
     \forall \psi_h \in \tilde{W}_h^{i,j}.
  \end{array}\nonumber
\end{eqnarray}
\textbf{Case $d=1$}
\vspace{2mm}

In that case, we write $U_h(t)=U^{n+1}_{0,h}+2\frac{t-t_{n+1/2}}{k_n}U_{1,h}^{n+1}$ and
$Q_h(t)=Q^{n+1}_{0,h}+2\frac{t-t_{n+1/2}}{k_n}Q_{1,h}^{n+1}$ on $I_n^i$,
$U_h^n=U_h(t_n)$, and the discrete problem reduces to find
$(U^{n+1}_{0,h},Q^{n+1}_{0,h})$ and $(U^{n+1}_{1,h},Q^{n+1}_{1,h})$ in $V^i_h \times \tilde{W}_h^{i,j}$
solution of the system
\begin{eqnarray}\label{eq:dGstronginterpd1-2}
  \begin{array}{l}
     \ds\frac{1}{k_n}(U_{0,h}^{n+1}+U_{1,h}^{n+1},v_h)+\tilde{a}(U_{0,h}^{n+1},v_h)
 +\int_\Gamma (Q_{0,h}^{n+1}- \frac{\av\cdot\nv}{2} U_{0,h}^{n+1})v_h \,d\sigma\\\ds \hspace{3cm}= (\frac{U^{n}}{k_n},v_h)
 +\frac{1}{k_n}\int_{I_n}(f,v_h)\,ds, \quad \forall v_h \in V^i_h, \\
         \ds\frac{3}{k_n}(-U_{0,h}^{n+1}+U_{1,h}^{n+1},w_h)+\tilde{a}(U_{1,h}^{n+1},w_h)
 +\int_\Gamma (Q_{0,h}^{n+1}- \frac{\av\cdot\nv}{2} U_{0,h}^{n+1}) w_h \,d\sigma\\\ds \hspace{3cm}= -(\frac{3 U^{n}}{k_n},w_h)
 +\frac{6}{k_n}\int_{I_n}\frac{s-t_{n+1/2}}{k_n}(f,w_h)\,ds, \quad \forall w_h \in V^i_h, \\[4mm]
   \ds\int_\Gamma\bigl(Q_{0,h}^{n+1}- \av\cdot\nv U_{0,h}^{n+1} +
    p\,U_{0,h}^{n+1} + \frac{q}{k_n}(U_{0,h}^{n+1}+U_{1,h}^{n+1})
    +q\divs{\Gamma}(\rv U_{0,h}^{n+1})\bigr)\psi_h\,d\sigma\\[3mm] \hspace{0.5cm}
    \ds+\int_\Gamma qs\grads{\Gamma}U_{0,h}^{n+1}\grads{\Gamma}\psi_h\,d\sigma
    = \int_\Gamma q\frac{U^{n}}{k_n}\psi_h \,d\sigma
     +\frac{1}{k_n}\int_{I_n}\int_\Gamma g \psi_h \,d\sigma\,ds,
    \quad \forall \psi_h \in \tilde{W}_h^{i,j},\\[3mm]
     \ds\int_\Gamma\bigl(Q_{1,h}^{n+1}- \av\cdot\nv U_{1,h}^{n+1} +p\,U_{1,h}^{n+1}
     + \frac{3q}{k_n}(-U_{0,h}^{n+1}+U_{1,h}^{n+1})
     + q\divs{\Gamma}(\rv U_{1,h}^{n+1})\bigr)\zeta_h\,d\sigma\\[3mm]
     \ds+\int_\Gamma qs\grads{\Gamma}U_{1,h}^{n+1}\grads{\Gamma}\zeta_h\,d\sigma
    = -\int_\Gamma q\frac{3U^{n}}{k_n} \zeta_h \,d\sigma
    +\frac{6}{k_n}\int_{I_n}\frac{s-t_{n+1/2}}{k_n}\int_\Gamma g\zeta_h \,d\sigma \,ds,
      \quad \forall \zeta_h \in \tilde{W}_h^{i,j}.
  \end{array}\nonumber
\end{eqnarray}

{\bf Transmission terms.}
Let $(U_{i,h}^{0}, Q_{i,h}^{0}) \in
{\cal P}_d(V_h^i,{\cal T}_i) \times {\cal P}_d(\tilde{W}_h^{i,j},{\cal T}_i)$
be a given initial guess, for $1 \le i \le I$.
Then, at iteration $k \ge 1$, we solve the subdomain problem in $\Omega_i$ :
\begin{eqnarray}\label{eq:algosdorder2discreteb}
  \begin{array}{l}
 \ds \frac{d\,}{dt}\,({\cal I}^iU_{i,h}^k, v_{i,h})_i+\tilde{a}_i(U_{i,h}^k,v_{i,h})_i\\
 \hspace{2cm}+\ds\int_{\Gamma_{i,j}} (Q_{i,h}^k-\frac{\av_i\cdot\nv_i}{2} U_{i,h}^k)v_{i,h} d\sigma
   = (P^if,v_{i,h})_i, \mbox{ in } (0,T), \ \forall v_{i,h} \in V^i_h,\\[5mm]
\ds\int_{\Gamma_{i,j}}\bigl(Q_{i,h}^k- \av_i\cdot\nv_i U_{i,h}^k
     + p_{i,j}\,U_{i,h}^k\bigr)\psi_h^{i,j}\,d\sigma\\
    \ds+ \int_{\Gamma_{i,j}}\bigl(q_{i,j}(\partial_t ({\cal I}^i U_{i,h}^k)
    + \divs{\Gamma_{i,j}}(\rv_{i,j} U_{i,h}^k))\psi_h^{i,j}
    + q_{i,j}s_{i,j}\grads{\Gamma_{i,j}}U_{i,h}^k\grads{\Gamma_{i,j}}\psi_h^{i,j}\bigr)\,d\sigma \\
    \hspace{2cm}\ds=
    P^i\tilde{g}_{h}((U_{j,h}^{k-1},Q_{j,h}^{k-1}),\psi_h^{i,j}),
    \mbox{ on } (0,T), \quad \forall \psi_h^{i,j} \in \tilde{W}_h^{i,j},j\in {\cal N}_i,
  \end{array}
\end{eqnarray}
with, for $k \ge 1$,
\begin{eqnarray}\label{eq:gkpsi}
  \begin{array}{l}
    \tilde{g}_{h}((U,Q),\psi):=\int_{\Gamma_{i,j}}\bigl(-Q
    +\av_{j}\cdot\nv_{j} U + p_{i,j}\,U\bigr)\psi\,d\sigma\\
   \ds + \int_{\Gamma_{i,j}}q_{i,j}(\partial_t ({\cal I}^j U)
    + \divs{\Gamma_{i,j}}(\rv_{i,j} U))\psi
    + q_{i,j}s_{i,j}\grads{\Gamma_{i,j}}U\grads{\Gamma_{i,j}}\psi\bigr)\,d\sigma.
  \end{array}
\end{eqnarray}
For the computation of the right-hand side in
\eqref{eq:algosdorder2discreteb}, we follow the same steps as in section \ref{sec:boundaryterms},
where we replace $\tilde{g}_{j,i}^k$ with
\begin{equation*}
\tilde{\gv}_{j,i,h}^k=\bigl(\tilde{g}_{h}((U_{j,h}^{k-1},Q_{j,h}^{k-1}),\psi_{1,h}^{i,j}),...,
      \tilde{g}_{h}((U_{j,h}^{k-1},Q_{j,h}^{k-1}),\psi_{m^{i,j},h}^{i,j})\bigr)^t,
\end{equation*}
and we replace $\tilde{G}^{i,k}_{m,\alpha}$ with $\tilde{G}^{i,k}_{m,\alpha,h} \in \tilde{W}_h^{i,j}$ solution of
\begin{equation}
\sum_{\alpha=0}^d \tilde{G}^{j,k}_{m,\alpha,h}\int_{I_m^j}\phi_{m,\alpha}^j\phi_{m,\beta}^j \,ds
=\int_{I_m^j} \tilde{\gv}^k_{j,i,h} \phi_{m,\beta}^j, \quad \beta \in \{0,...,d\}.
\end{equation}
For the computation of $\tilde{\gv}_{j,i,h}^k$, we write
$(U_{j,h}^{k-1})_{|_{\Gamma_{i,j}}}=\sum_{l=1}^{n^{j,i}}u_{j,l}^h\chi^{j,i}_{l,h}$, and
$Q_{j,h}^{k-1}=\sum_{\ell=1}^{m^{j,i}}z_{j,\ell}^h\psi^{j,i}_{\ell,h}$, and introduce
$\vecb{Q}_{j,h}^{k-1}=(z_{j,1}^h,...,z_{j,m^{j,i}}^h)^t$,
$\vecb{U}_{j,h}^{k-1}=(u_{j,1}^h,...,u_{j,n^{j,i}}^h)^t$,
and the projection
matrices, for $1\le k \le m^{i,j},\, 1\le l \le m^{j,i}$, and $1\le \ell \le n^{j,i}$,
\begin{multline*}
(\widetilde{\mathbb{M}}_h^{i,j})_{k,l}=\int_{\Gamma_{i,j}}\psi^{i,j}_{k,h}\psi^{j,i}_{l,h}\,d\sigma,
\ (\mathbb{M}_{h}^{i,j})_{k,\ell}=\int_{\Gamma_{i,j}}\psi^{i,j}_{k,h}\chi^{j,i}_{\ell,h}\,d\sigma,\\
(\mathbb{M}_{b,h}^{i,j})_{k,\ell}=\int_{\Gamma_{i,j}}(-\av_j\cdot\nv_i+p_{i,j})
\psi^{i,j}_{k,h}\chi^{j,i}_{\ell,h}\,d\sigma,\\
(\mathbb{B}_{r,h}^{i,j})_{k,\ell}=\int_{\Gamma_{i,j}}\divs{\Gamma_{i,j}}(\rv_{i,j} \chi^{j,i}_{\ell,h})\psi^{i,j}_{k,h}
\,d\sigma,
\ (\mathbb{K}_{s,h}^{i,j})_{k,\ell}=\int_{\Gamma_{i,j}}q_{i,j}s_{i,j}
   \grads{\Gamma_{i,j}}\chi^{j,i}_{\ell,h}\grads{\Gamma_{i,j}} \psi^{i,j}_{k,h}
\,d\sigma.
\end{multline*}
Then
\begin{equation}
\tilde{\gv}_{j,i,h}^k=-\widetilde{\mathbb{M}}_h^{i,j}\vecb{Q}_{j,h}^{k-1}
+\partial_t({\cal I}^j \mathbb{M}_{h}^{i,j} \vecb{U}_{j,h}^{k-1})
+(\mathbb{M}_{b,h}^{i,j}+ \mathbb{B}_{r,h}^{i,j}+\mathbb{K}_{s,h}^{i,j})\vecb{U}_{j,h}^{k-1}.
\end{equation}
The projection matrices are computed using the projection algorithm in \cite{gander:2005:NZF}.
%
\section{Numerical Results}
\label{sec:NumRes}
We have implemented the algorithm with $d=1$ and $\PP_1$
finite elements in space in each subdomain. Time windows are used in
order to reduce the number of iterations of the algorithm.  For the free parameters defining $S_{i,j}$ and $\tilde{S}_{i,j}$, we chose $r_{i,j}$ to be the tangential component of the advection, $s_{i,j}$ the value of the diffusion in the domain $\Omega_j$. The
optimized parameters $p_{i,j}$ and $q_{i,j}$ are constant along the interface.  They correspond to a mean value of the parameters obtained by a numerical
optimization of the convergence factor \cite{gander:2007:SWR}.
%

We first give an example of a multidomain solution with discontinuous
variable diffusion, for two subdomains and one time window. The
advection velocity is also discontinuous, taken normal to the
interface in one subdomain, and tangential to the interface in the
other subdomain. The latter case of a flow tangential to the interface
is difficult when the interface conditions are not related to the
convergence factor of the domain decomposition method (see for example
\cite{japhet:1998:MDD}).

The physical domain is
$\Omega=(0,1) \times (0,2)$, the final time is $T=1$.  The initial
value is $u_0=0.25e^{-100((x-0.55)^2+(y-1.7)^2)}$ and the right-hand
side is $f=0$.  The domain $\Omega \times (0,2)$ is split into two subdomains
$\Omega_1=(0,0.5) \times (0,2)$ and $\Omega_2=(0.5,1) \times (0,2)$.
The reaction $c$ is zero, the advection and diffusion coefficients are
$\av_1=(0,-1)$, $\nu_1=0.001\sqrt{y}$, and $\av_2=(-0.1,0)$,
$\nu_2=0.1\sin(xy)$.  The mesh size over the interface and time step
in $\Omega_1$ are $h_1=1/32$ and $k_1=1/128$, while in $\Omega_2$,
$h_2=1/24$ and $k_2=1/94$.  In Figure \ref{fig:1}, we observe, at
final time $T=1$, that the approximate solution computed using 3
iterations (right figure) is close to the variational solution
computed in one time window on a time conforming finergrid (left
figure).
\begin{figure}[H]
\centering
\includegraphics[width=0.49\textwidth]{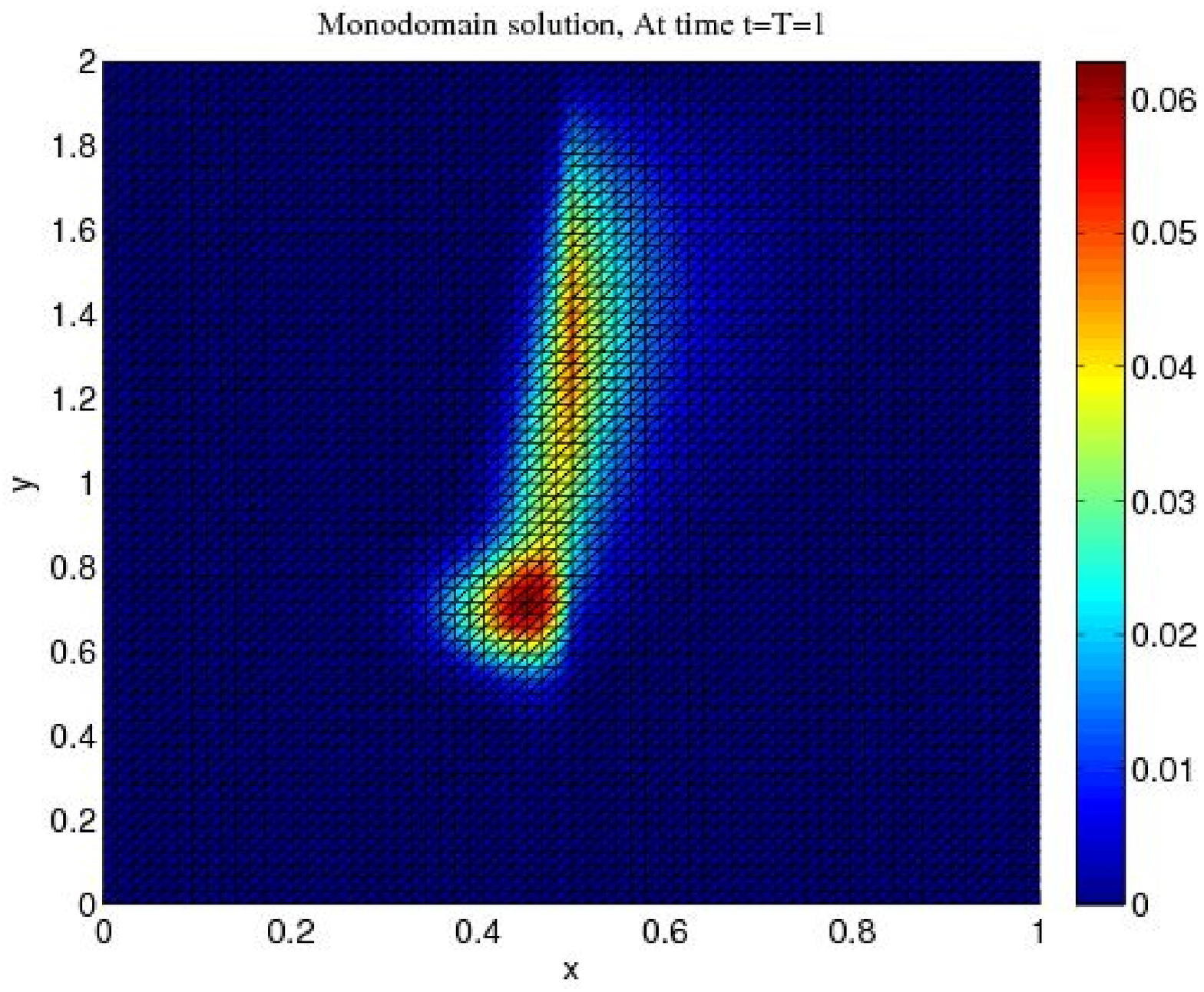}
\hfill
\includegraphics[width=0.49\textwidth]{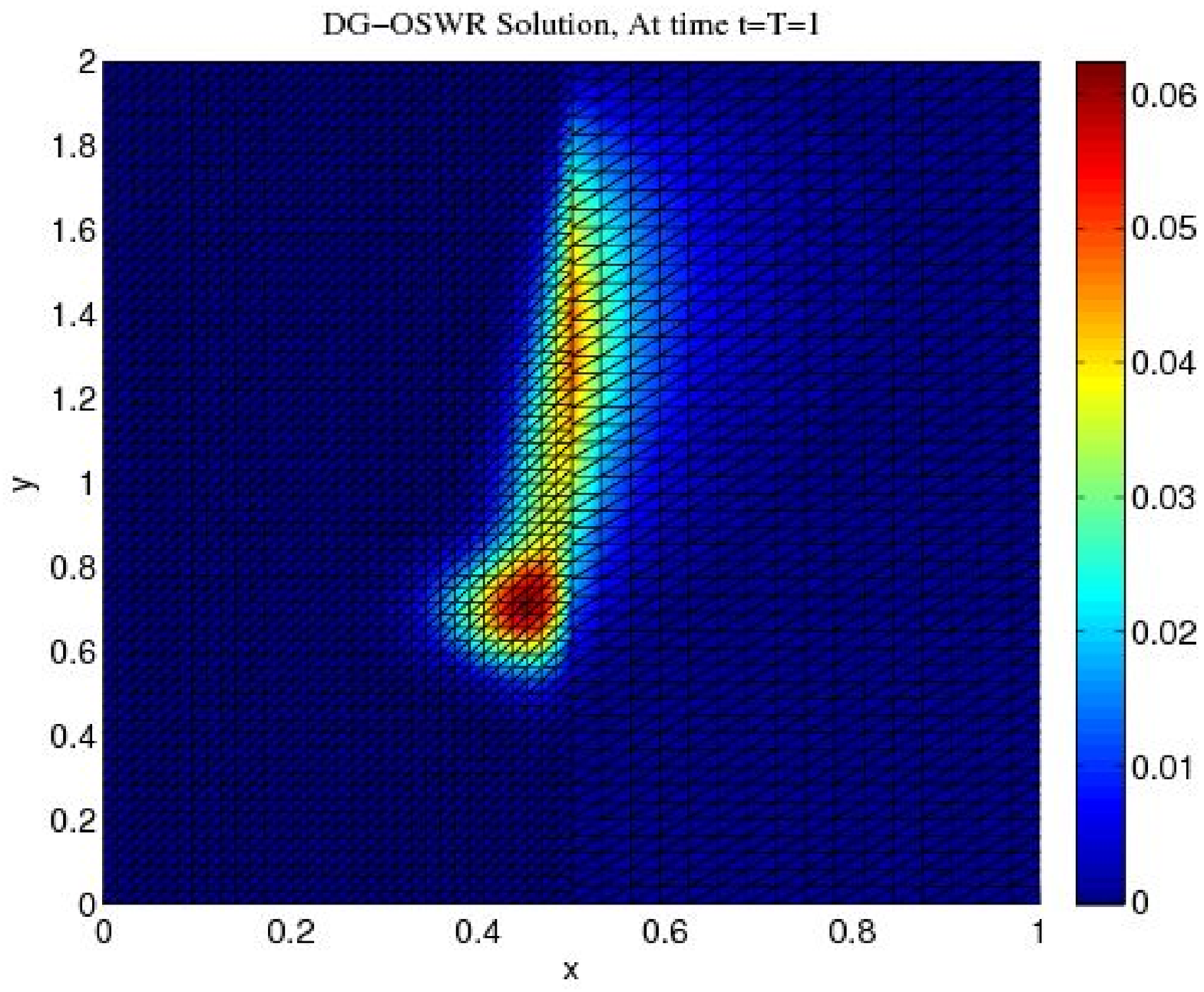}
\caption{Variational  (left) and nonconforming (right) DG-OSWR solutions}
\label{fig:1}       
\end{figure}

We analyze now the precision in time. The space mesh
is conforming and the converged solution is such that the residual is
smaller than $10^{-8}$. We compute a variational reference solution
on a time grid with 4096 time steps. The nonconforming solutions are
interpolated on the previous grid to compute the error.
We start with a time grid with 128 time steps for the left domain and
94 time steps for the right domain.
Thereafter the time step is divided by 2 several times. Figure
\ref{fig:2} shows the norms of the error in $L^{\infty}(I;L^2(\Omega_i))$
 versus the number of refinements, for both subdomains.
First we observe the order $2$ in time for the nonconforming
case. This fits the theoretical estimates, even though we have theoretical
results only for Robin transmission conditions.
Moreover, the error obtained in the nonconforming case, in the subdomain
where the grid is finer, is nearly the same as the error obtained
in the conforming finer case.
\begin{figure}[H]
\centering
\includegraphics[width=8cm]{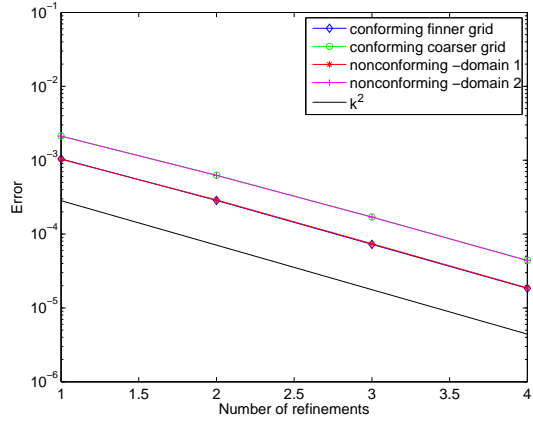}
\caption{Error between variational and DG-OSWR solutions versus the refinement in time}
\label{fig:2}       
\end{figure}
The computations are done using Order 2 transmission conditions. Indeed, the error between the multidomain and the variational solutions decrease much faster with the Order 2 transmissions conditions than with the Robin transmissions conditions as we can see in Figure \ref{fig:3}, in the conforming case.
\begin{figure}[H]
\centering
\includegraphics[width=9cm]{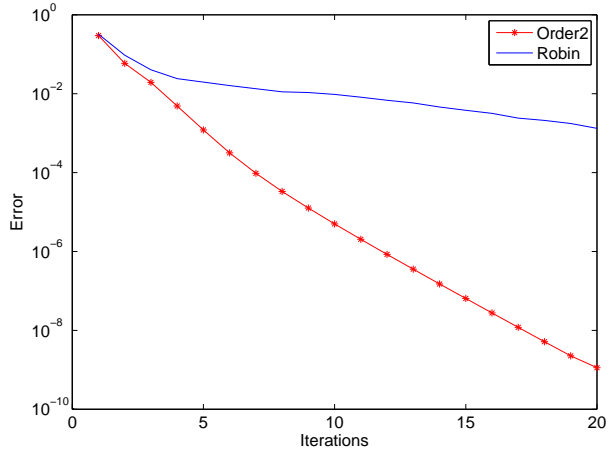}
\caption{Convergence history for different transmission conditions}
\label{fig:3}       
\end{figure}

We now consider the advection-diffusion equation with discontinuous porosity $\omega$:
$$\displaystyle \omega \partial_t u +\nabla
\cdot({\av}u-\nu \nabla u)=0. $$
The physical domain is $\Omega=(0,1) \times (0,2)$, the final time is $T=1.5$. $\Omega$ is split into two subdomains. The interface  $\Gamma$ is parametrized  with a Hermite polynomial $({1 \over 2}+((2s-1)^3+2(2s-1)^2+(2s-1))\displaystyle{1\!\!1_{s \le {1 \over 2}}}+((2s-1)^3-2(2s-1)^2+(2s-1))
\displaystyle{1\!\!1_{s \ge {1 \over 2}}}), \, 0<s<1$, see Figure \ref{fig:4}.
The advection, diffusion and porosity
coefficients are \\\
$\av_1=(-sin({\pi \over 2}(y-1))cos(\pi(x-{1 \over
2})),3cos({\pi \over 2}(y-1))sin(\pi(x-{1 \over 2})))$, $\nu_1=0.003$,
$\omega_1=0.1$,\\
$\av_2=\av_1$, $\nu_2=0.01$, $\omega_2=1$.

\begin{figure}[H]
\centering
\includegraphics[width=0.8\textwidth]{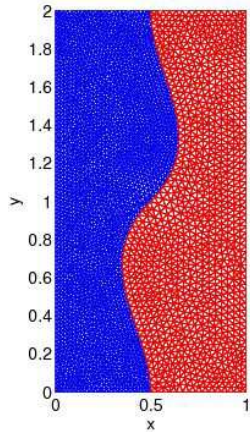}
\caption{Domains $\Omega_1$ (left) and $\Omega_2$ (right)}
\label{fig:4}       
\end{figure}

\noindent
We first consider a conforming grid in space. The time step in $\Omega_1$ is $k_1=1/180$, while in $\Omega_2$, $k_2=1/100$. In Figure \ref{fig:5}, we observe, at final time $T=1.5$, the approximate solution computed using ten time windows and 5 iterations in each time window. It is close to the variational solution computed in one time window on the conforming finer space-time grid as shown on the error, in Figure \ref{fig:5b}.
\begin{tabular}{cc}
\begin{minipage}{0.45\textwidth}
\centering
\begin{figure}[H]
\hspace{-2cm}
\includegraphics[scale=0.45]{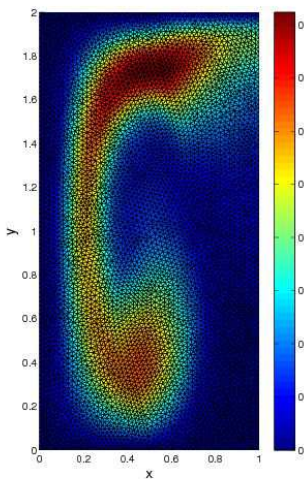}
\caption{DG-OSWR solution after 10 time
windows and 5 iterations per window}
\label{fig:5}       
\end{figure}
\end{minipage}
&
\begin{minipage}{0.45\textwidth}
\centering
\begin{figure}[H]
\hspace{-2cm}
\includegraphics[scale=0.45]{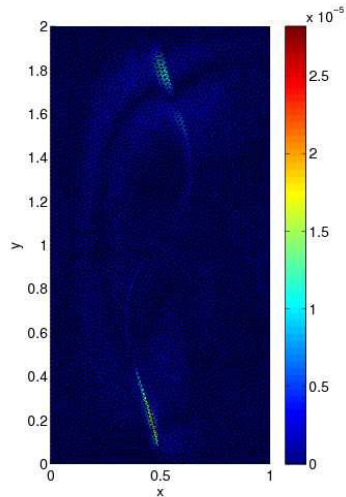}
\caption{Error with variational solution after 10 time
windows and 5 iterations per window}
\label{fig:5b}      
\end{figure}
\end{minipage}
\end{tabular}

\leavevmode\par
We analyze in Figure \ref{fig:6} the precision versus the time step. The converged solution is such that the residual is smaller than $10^{-12}$. A variational reference solution is computed on a time grid with 7680 time steps. The time nonconforming solutions are interpolated on the previous grid to compute the error.  We start with a time grid with 120 time steps for the left domain and 26 time steps for the right domain and divide by 2 the time steps several times.  Figure \ref{fig:6} shows the norms of the error in $L^{\infty}(I;L^2(\Omega_i))$ versus the time steps, for both subdomains.
\begin{figure}[H]
\centering
\includegraphics[scale=0.35]{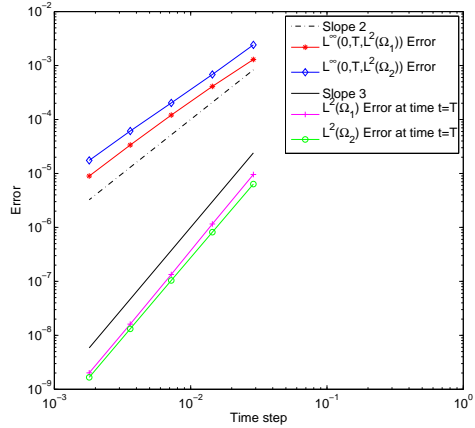}
\caption{Error curves versus the refinement in time}
\label{fig:6}       
\end{figure}

We observe the order $2$ in time for the nonconforming case that fits the theoretical estimates. In Figure \ref{fig:6} we show also the norms of the error in $L^2(\Omega_i)$ at final time $t=T$ versus the time steps, for both subdomains. We observe the order $3$ for the time nonconforming  case. This corresponds to the superconvergence behavior described in \cite{johnson:1995:dgparab2}.

We now consider nonconforming grids in space as well. The mesh size and time step in $\Omega_1$ are $h_1=0.032$ and $k_1=1/120$, while in $\Omega_2$, $h_2=0.048$ and $k_2=1/26$.  In Figure \ref{fig:7} we observe, at final time $T=1.5$, that the approximate solution computed using 5 iterations in one time window is close to the solution computed with the conformal in space grid in Figure \ref{fig:5}, left. In Figure \ref{fig:8} and \ref{fig:8b} we observe the precision versus the mesh size. The converged solution is such that the residual is smaller than $10^{-12}$. A variational reference solution is computed on a time grid with 960 time steps and a space grid with mesh size $h=3.5 \,10^{-3}$.  The space-time nonconforming solutions are interpolated on the previous grid to compute the error.  We start with a time grid with 60 time steps and a mesh size $h_1=0.056$ for the left domain and 20 time steps a mesh size $h_2=0.11$ for the right domain and divide by 2 the time step and mesh size several times. Figure \ref{fig:8} shows the norms of the error in $L^2(I;L^2(\Omega_i))$ versus the time steps, for both subdomains. We observe the order $2$ for the nonconforming space-time case, even though we have theoretical results only for the time semi-discrete case. Figure \ref{fig:8b} displays the norms of the error in $L^2(\Omega_i)$ and in $H^1(\Omega_i)$ at final time $t=T$ versus the mesh size, for both subdomains. We observe the order $2$ for the $L^2$ error, and the order $1$ for the $H^1$ error for the nonconforming space-time case.
%
%
\begin{figure}[H]
\centering
\includegraphics[scale=0.6]{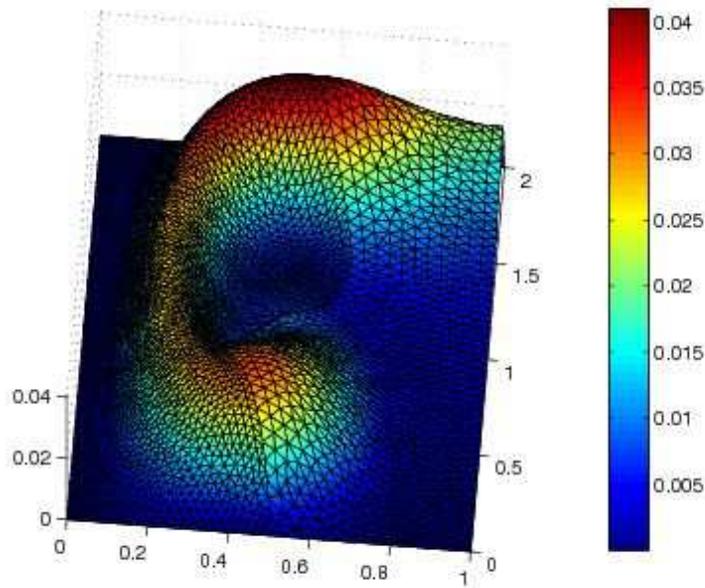}
\caption{DG-OSWR solution after 5
  iterations, in one time window}
\label{fig:7}       
\end{figure}

\begin{tabular}{cc}
\begin{minipage}{0.45\textwidth}
\begin{figure}[H]
\hspace{-20mm}
\includegraphics[width=7.49cm]{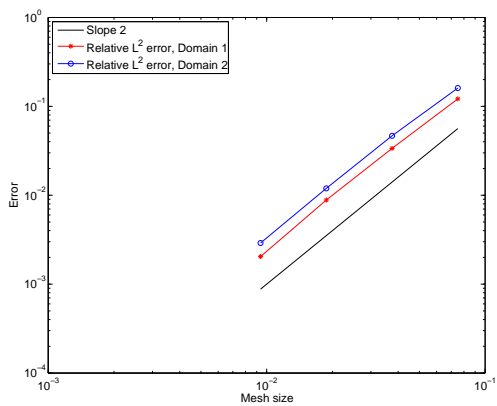}
\caption{Relative $L^{2}$ error in time and space}
\label{fig:8}       
\end{figure}
\end{minipage}
&
\begin{minipage}{0.45\textwidth}
\begin{figure}[H]
\hspace{-10mm}
\includegraphics[width=7.49cm]{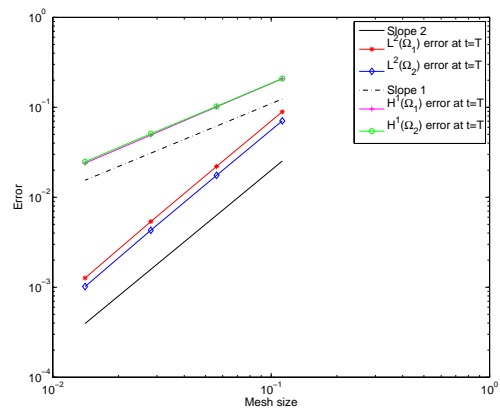}
\caption{$L^{2}$ and $H^{1}$ errors at the final time}
\label{fig:8b}       
\end{figure}
\end{minipage}
\end{tabular}
\section{Conclusion}
We have proposed a new numerical method to solve parabolic equations with discontinuous coefficients. It relies on the
splitting of the time interval into time windows, in which a few
iterations of an optimized Schwarz waveform relaxation algorithm are performed by a discontinuous
Galerkin method in time, with non conforming projection between space-time grids on the interfaces. We have shown theoretically in the Robin case that the method preserves the order of the discontinuous Galerkin method. Numerical estimates of the $L^2(I;L^2(\Omega_i))$ error and the $H^1$ error at final time have shown that the method preserves the order of the space nonconforming scheme as well. The analysis of the
fully discrete scheme will be done in a further work.

\vspace{-0.5cm}
%

 \bibliographystyle{plain}
\bibliography{paper} %



\end{document}